\documentclass[UTF-8,reqno]{amsart}
\usepackage{enumerate, bbm}
\usepackage{amssymb,url,color, booktabs}
\usepackage{tikz}
\usetikzlibrary{arrows,patterns,calc}
\usepackage{float}
\usepackage{graphicx}
\usepackage{mathrsfs}
\usepackage{color}
\usepackage[colorlinks=true]{hyperref}
\hypersetup{
    linkcolor=blue,          
    citecolor=red,        
    filecolor=blue,      
    urlcolor=cyan
}

\usepackage{color}
\usepackage{ulem}
\usepackage[thicklines]{cancel}

\usepackage{soul} 
\usepackage{color, xcolor} 

\definecolor{MyDarkBlue}{cmyk}{0.8,0.3,0.8,0.4}
\definecolor{yellow}{rgb}{0.99,0.99,0.70}
\definecolor{white}{rgb}{1.0,1.0,1.0}
\definecolor{black}{rgb}{0.00,0.00,0.00}

\newcommand{\red}{\color{red}}

\numberwithin{equation}{section}

\newcommand{\be}{\begin{eqnarray}}
\newcommand{\ee}{\end{eqnarray}}
\newcommand{\ce}{\begin{eqnarray*}}
\newcommand{\de}{\end{eqnarray*}}
\newtheorem{theorem}{Theorem}[section]
\newtheorem{lemma}[theorem]{Lemma}
\newtheorem{remark}[theorem]{Remark}
\newtheorem{definition}[theorem]{Definition}
\newtheorem{proposition}[theorem]{Proposition}
\newtheorem{Examples}[theorem]{Example}
\newtheorem{corollary}[theorem]{Corollary}

\makeatletter
\@namedef{subjclassname@2020}{\textup{2020} Mathematics Subject Classification}
\makeatother

\def\var{{\mathrm{var}}}

\def\eps{\varepsilon}

\def\e{\mathrm{e}}

\def\p{\partial}

\def\[{{\Big[}}
\def\]{{\Big]}}
\def\<{{\langle}}
\def\>{{\rangle}}
\def\({{\big(}}
\def\){{\big)}}

\def\sgn{\mbox{\rm sgn}}

\def\dif{{\mathord{{\rm d}}}}

\def\min{{\mathord{{\rm min}}}}

\def\bb2{{\boldsymbol{2}}}
\def\no{\nonumber}
\def\={&\!\!=\!\!&}

\def\bC{{\mathbf C}}

\def\cA{{\mathcal A}}
\def\cB{{\mathcal B}}

\def\cD{{\mathcal D}}

\def\cF{{\mathcal F}}

\def\cL{{\mathcal L}}
\def\cM{{\mathcal M}}

\def\cP{{\mathcal P}}

\def\cR{{\mathcal R}}

\def\cW{{\mathcal W}}

\def\mE{{\mathbb E}}

\def\mI{{\mathbb I}}

\def\mL{{\mathbb L}}

\def\mN{{\mathbb N}}

\def\mP{{\mathbb P}}

\def\mR{{\mathbb R}}

\def\bP{{\mathbf P}}

\def\bD{{\mathbf D}}

\def\b1{{\mathbbm 1}}

\def\sB{{\mathscr B}}

\def\sI{{\mathscr I}}
\def\sJ{{\mathscr J}}

\def\sL{{\mathscr L}}

\def\sS{{\mathscr S}}

\def\geq{\geqslant}
\def\leq{\leqslant}
\def\ge{\geqslant}
\def\le{\leqslant}

\def\supp{{\mathrm{supp}}}
\def\var{{\mathrm{var}}}

\def\eps{\varepsilon}

\def\e{\mathrm{e}}

\def\p{\partial}

\def\[{{\Big[}}
\def\]{{\Big]}}
\def\<{{\langle}}
\def\>{{\rangle}}

\def\sgn{\mbox{\rm sgn}}

\def\dif{{\mathord{{\rm d}}}}

\def\min{{\mathord{{\rm min}}}}

\def\no{\nonumber}
\def\={&\!\!=\!\!&}
\def\bt{\begin{theorem}}
\def\et{\end{theorem}}
\def\bl{\begin{lemma}}
\def\el{\end{lemma}}
\def\br{\begin{remark}}
\def\er{\end{remark}}
\def\bd{\begin{definition}}
\def\ed{\end{definition}}
\def\bp{\begin{proposition}}
\def\ep{\end{proposition}}
\def\bc{\begin{corollary}}
\def\ec{\end{corollary}}

\def\geq{\geqslant}
\def\leq{\leqslant}
\def\ge{\geqslant}
\def\le{\leqslant}

 \def\R{\mathbb R}
 \def\R{\mathbb R}    
\def\N{\mathbb N}  
 
\def\<{\langle} \def\>{\rangle}

\def\x{{\mathbf x}}

\def\bb2{{\boldsymbol{2}}}
\def\bbb1{\boldsymbol{1}}

\def\no{\nonumber}
\def\={&\!\!=\!\!&}

\allowdisplaybreaks

\title[Averaging principle]
{Averaging principle for SDEs with singular drifts
driven by $\alpha$-stable processes}

\author{Mengyu Cheng, Zimo Hao and Xicheng Zhang}

\address{Mengyu Cheng:
School of Mathematics and Statistics,
Beijing Institute of Technology, Bejing 100081, China}
\email{mengyu.cheng@hotmail.com; mcheng@bit.edu.cn}

\address{Zimo Hao:
Fakult\"at f\"ur Mathematik, Universit\"at Bielefeld,
33615, Bielefeld, Germany}
\email{zhao@math.uni-bielefeld.de}

\address{Xicheng Zhang:
School of Mathematics and Statistics, Beijing Institute of Technology, Beijing 100081, China}
\email{XichengZhang@gmail.com}

\thanks{
The first author was supported by NNSFC grant of China (No. 12301223) and
the China Postdoctoral Science Foundation (Nos. 2023M740260, GZB20230937 and 2024T171119).
The second and third authors were supported by National Key R\&D program of China
(No. 2023YFA1010103) and NNSFC grant of China (No. 12131019)
and the DFG through the CRC 1283 ``Taming uncertainty and profiting from randomness
and low regularity in analysis, stochastics and their applications''.}



\begin{document}

\begin{abstract}
  In this paper, we investigate the convergence rate of the averaging principle
  for stochastic differential equations (SDEs) with $\beta$-H\"older drift
  driven by $\alpha$-stable processes.
  More specifically, we first derive the Schauder estimate
  for nonlocal partial differential equations (PDEs) associated with
  the aforementioned SDEs, within the framework of Besov-H\"older spaces.
  Then we consider the case where $(\alpha,\beta)\in(0,2)\times(1-\tfrac{\alpha}{2},1)$.
  Using the Schauder estimate,
  we establish the strong convergence rate for the averaging principle.
  In particular, under suitable conditions we obtain the optimal rate of strong convergence when $(\alpha,\beta)\in(\tfrac{2}{3},1]\times(2-\tfrac{3\alpha}{2},1)
\cup(1,2)\times(\tfrac{\alpha}{2},1)$.
  Furthermore, when $(\alpha,\beta)\in(0,1]\times(1-\alpha,1-\tfrac{\alpha}{2}]
  \cup(1,2)\times(\tfrac{1-\alpha}{2},1-\tfrac{\alpha}{2}]$,
  we show the convergence of the martingale solutions of original systems to
  that of the averaged equation.
  When $\alpha\in(1,2)$, the drift can be a distribution.

\bigskip


\end{abstract}

\maketitle \rm


\section{Introduction}

Consider the perturbation problem of ODE in the standard form:
\begin{equation}\label{Ineq01}
\dot{y} = \varepsilon f(t,y), \quad y(0)=y_0 \in \mathbb{R}^d,
\end{equation}
where $f: \mathbb{R} \times \mathbb{R}^d \rightarrow \mathbb{R}^d$ is a measurable function and
the small parameter $0 < \varepsilon \ll 1$ characterizes the magnitude of the perturbation.
It is worth noting that the evolution of \eqref{Ineq01} is negligible
on every finite time interval $[0,T]$,
but it may become significant on time intervals of the form $[0,\tfrac{T}{\varepsilon}]$.
The idea of averaging principle is to find the following time-independent system,
that is called averaged system,
for describing the evolution of \eqref{Ineq01} over a long time interval $[0,T\varepsilon^{-1}]$:
\begin{equation}\label{Ineq0318}
\dot{\bar{y}} = \varepsilon \bar{f}(\bar{y}), \quad \bar{y}(0) = y_0 \in \mathbb{R}^d,
\end{equation}
where
\begin{equation}\label{0318:01}
\bar{f}(y) := \lim_{T \rightarrow \infty} \frac{1}{T} \int_0^{T} f(s,y) \, \mathrm{d}s,
\end{equation}
and the limit in \eqref{0318:01} is taken uniformly on any bounded subset of $\mathbb{R}^d$.
Such an $f$ is called a KBM-{\it vector field} (KBM stands for Krylov, Bogolyubov and Mitropolsky);
see e.g. \cite{SVM07}.

The averaging principle, initially proposed by Krylov,
Bogolyubov and Mitropolskii \cite{BM1961}, reformulated by Hale \cite{Hale1963},
further developed by Volosov \cite{Volo1962},
and presented in a geometric form by Arnold \cite{Arnold12},
has played significant role in the field of multi-scale systems and perturbation theory.
Astrom \cite{Astrom83, Astrom84} later applied
this method to analyze the stability of deterministic adaptive systems.
Averaging principle proves invaluable in evaluating the stability of adaptive systems,
particularly in the presence of unmodeled dynamics,
and shedding light on mechanisms of instability. Beyond its application in stability problems,
the averaging principle serves as a powerful approximation method that replace
a system of nonautonomous differential equations
by an autonomous system; see \cite{AKN06, FBS1986, SB89} for example.
Therefore, the averaging principle is productive and widely used in mathematical physics.

Real-world problems are frequently influenced by noise perturbations arising from
surrounding environments or intrinsic uncertainties.
Consequently, more realistic models should often take the fluctuation
or noise into consideration.
Following Khasmiskii's work \cite{Khas1968},
the averaging principle has also been extensively studied for SDEs
(see e.g. \cite{Vere90, MSV91, BK04, KY04, FW06, XDX12, LRSX20, HL20, SXX22, CLR23b})
and SPDEs (see e.g. \cite{Cerr09, CF09, WR2012, DW2014, DSXZ18, CL23, CLR23a, RXY23}).

Interestingly, noise can sometimes introduce
beneficial effects into the system. It is important to note
that \eqref{Ineq01} may become ill-posed when $b$ is only
H\"older continuous. However, the system can present well-posed after the introduction of noise.
Based on the aforementioned motivations,
in this paper we consider the following singular SDEs driven by $\alpha$-stable process:
\begin{equation}\label{0318:02}
\mathrm{d}Y_t^\varepsilon = \varepsilon b(t,Y_t^\varepsilon) \, \mathrm{d}t
+ \varepsilon^{\frac{1}{\alpha}} \sigma(t,Y_t^\varepsilon) \, \mathrm{d}L_t^{(\alpha)},
\end{equation}
where $b: [0,T] \times \mathbb{R}^d \rightarrow \mathbb{R}^d$
belongs to some H\"older space $\bC^\beta$
(see Section \ref{BS} for definition),
$\{L_t^{(\alpha)}, t \geq 0\}$ is an $\alpha$-stable process,
and $\sigma: [0,T] \times \mathbb{R}^d \rightarrow \mathbb{R}^d \otimes \mathbb{R}^d$
is bounded and nondegenerate.
In contrast to the aforementioned works on SDEs,
our system \eqref{0318:02} has H\"older continuous drift $b$.

When $(\alpha, \beta) \in (0,2) \times \left(1 - \frac{\alpha}{2}, 1\right) =: \mathcal{A}_1$,
for all $0 < \varepsilon \ll 1$ and $x \in \mathbb{R}^d$
there exists a unique strong solution $\{X_t^\varepsilon, t \geq 0\}$ to \eqref{0318:02} with $X_0^\varepsilon = x$;
see e.g. \cite{CZZ22}.
If
$$
(\alpha, \beta) \in (0,1] \times \left(1 - \alpha, 1 - \tfrac{\alpha}{2}\right]
\cup (1,2) \times \left(\tfrac{1 - \alpha}{2}, 1 - \tfrac{\alpha}{2}\right] =: \mathcal{A}_2
$$
then for any $0 < \varepsilon \ll 1$ and $x \in \mathbb{R}^d$, there exists
a unique martingale solution $\{\mathbb{P}_x^\varepsilon(t), t \geq 0\}$ to \eqref{0318:02}
(as defined in Definition \ref{MSolution}); see \cite{LZ22, WH23} for example.
Moreover, it has been shown in \cite{TTW74} that
for any $(\alpha, \beta) \in (0,1) \times (0, 1 - \alpha)$, there exists a
bounded $\beta$-H\"older continuous $b$ such that both strong and weak uniqueness
for \eqref{0318:02} fails.
Therefore, we consider the strong convergence (respectively, weak convergence)
for the averaging principle to \eqref{0318:02} when $(\alpha, \beta) \in \mathcal{A}_1$
(respectively, $(\alpha, \beta) \in \mathcal{A}_2$).

To focus on the evolution of \eqref{0318:02} over times of order $\frac{1}{\varepsilon}$,
we rescale the solution $Y_t^\varepsilon$ by defining $\widetilde{Y}_t^\varepsilon := Y_{t/\varepsilon}^\varepsilon$
for all $t \in [0,T]$ with $0 < \varepsilon \ll 1$. Subsequently, it becomes evident that
$\widetilde{Y}_t^\varepsilon$ shares the same distribution as $X_t^\varepsilon$, which solves the SDE:
\begin{align}\label{Main01}
\mathrm{d}X^\varepsilon_t = b\left(\frac{t}{\varepsilon}, X^\varepsilon_t\right) \, \mathrm{d}t
+ \sigma\left(\frac{t}{\varepsilon}, X^\varepsilon_t\right) \, \mathrm{d}L_t^{(\alpha)},
\quad 0 < \varepsilon \ll 1.
\end{align}
Therefore, the primary objective of this paper
is to show the convergence of $X_t^\varepsilon$ to
$\bar{X}_t$ in both the strong and weak sense as
the time scale $\varepsilon$ goes to zero, while also obtain the convergence rate,
which plays a crucial role in constructing efficient numerical schemes
and homogenization in PDE (see e.g. \cite{PV01, PV03, KY04, WR2012}).
Here $\bar{X}_t$ solves the so-called averaged system:
\begin{align}\label{0318:03}
\mathrm{d}\bar{X}_t = \bar{b}(\bar{X}_t) \, \mathrm{d}t + \bar{\sigma}(\bar{X}_t) \, \mathrm{d}L_t^{(\alpha)},
\end{align}
where $\bar{b}$ and $\bar{\sigma}$ can be formally defined by
$$
\bar{b}(x) := \lim_{T \rightarrow \infty} \frac{1}{T} \int_0^T b(s,x) \, \mathrm{d}s,
\quad
\lim_{T \rightarrow \infty} \frac{1}{T} \int_0^T \left|\sigma(s,x) - \bar{\sigma}(x)\right|^2 \, \mathrm{d}s = 0,
\quad \forall x \in \mathbb{R}^d.
$$
More precisely, in the case where
$(\alpha, \beta) \in \mathcal{A}_1$,
we obtain the strong convergence rate; see Theorem \ref{Th1}.
In particular, we have the optimal strong convergence rate provided that
$$
(\alpha, \beta) \in \left(\tfrac{2}{3}, 1\right] \times
\left(2 - \tfrac{3\alpha}{2}, 1\right)
\cup (1,2) \times \left(\tfrac{\alpha}{2}, 1\right) =: \mathcal{A}_0;
$$
see Example \ref{Ex1} and Remark \ref{Rk1} (2)-(i).
We finally point out that our convergence rate is faster than $\frac{1}{3}-$ obtained in \cite{CHR24};
see Remark \ref{Rk1} for details.

It is worth mentioning that our result presents a unified framework applicable to the following
slow-fast systems:
\begin{equation}\label{0820:01}
\dif X_t^\eps=f(X_t^\eps,Y_t^\eps)\dif t+\sigma(X_t^\eps)\dif L_t^{(\alpha)},
\end{equation}
where $f:\R^n\times\R^m\rightarrow\R^n$ belongs to $L^\infty(\R^m;\bC^\beta(\R^n))$,
the fast process $Y_t^\eps:=Y_{t/\eps}$, $t\geq0$,
which is independent of $\{L_t^{(\alpha)},t\geq0\}$
and satisfies a kind of strong law of large numbers. Roughly speaking,
if there is an $\bar{f}:\R^n\rightarrow\R^n$ such that for all $x\in\R^n$,
\begin{equation*}\label{0820:02}
\lim_{T\rightarrow\infty}\frac{1}{T}\int_0^Tf(x,Y_t)\dif t
=\bar{f}(x) \quad \mP-{\rm a.s.},
\end{equation*}
then $X_t^\eps$ strongly converges to $\bar{X}_t$, where $\bar{X}_t$ is the solution to
\begin{equation*}
\dif \bar{X}_t=\bar{f}(\bar{X}_t)\dif t+\sigma(\bar{X}_t)\dif L_t^{(\alpha)};
\end{equation*}
see Theorem \ref{thSF} for more details. To illustrate this result,
we provide Example \ref{Ex3} below. In this example, we assume that
$f\in L^\infty(\R^m_y;\bC_b^{\beta_1})\cap L^\infty(\R^n_x;\bC_b^{\beta_2})$
with $(\beta_1,\beta_2)\in(1-\tfrac{\alpha}{2})\times(0,1]$ and
$\{Y_t,t\geq0\}$ is the solution to
\begin{align*}\label{0820:03}
\dif Y_t=B(Y_t)\dif t+\dif W_t,
\end{align*}
where $\{W_t,t\geq0\}$ is an $\R^m$-valued Brownian motion and
$B:\R^m\rightarrow\R^m$ satisfies
$$
\lim_{|y|\rightarrow\infty}\langle B(y),y\rangle=-\infty.
$$
By Theorem \ref{thSF}, we show that for any $\kappa>0$,
\begin{equation*}
\mE\left(\sup_{t\in[0,T]}|X_t^\eps-\bar{X}_t|\right)
\leq C\eps^{\frac{(1\wedge\alpha)(\frac12-\kappa)}{(1\wedge\alpha)+\delta_1}},
\end{equation*}
where $\delta_1:=[(\tfrac{\alpha}{2}-\beta_1)\vee(2-\tfrac{3\alpha}{2}-\beta_1)+\iota]\vee0$
for any $\iota>0$.

Recall that there are few works on the averaging principle for slow-fast SDEs
with non-smooth coefficients; see \cite{Vere90} for weak convergence without rate
and \cite{RX21, Xie23} for strong/weak convergence rate.
In \cite{RX21}, they obtain the strong convergence rate $\tfrac{\beta_1 \wedge 1}{2}$
for SDEs driven by Brownian motion.
Regarding the case where the slow equation driven by an $\alpha$-stable process,
when $\alpha \in [1,2)$, the strong convergence rate $\frac{\beta_1}{\alpha} \wedge \frac{1}{2}$
is obtained in \cite{Xie23},
and the optimal rate can be reached if $\beta_1 \in [\alpha/2,1)$.
However, in Example \ref{Ex3} we can consider the case where $\alpha \in (0,2)$
and obtain the strong convergence rate $\tfrac12-$ provided $(\alpha, \beta) \in \mathcal{A}_0$.
Compared with \cite{RX21, Xie23}, the rate obtained in Example \ref{Ex3}
has a lower bound $\tfrac{\alpha}{2}\wedge\tfrac{1}{2\alpha}$,
which is independent of $\beta_1$.
This observation is different from previous results.
Moreover, as far as we know, there is no existing work on averaging
principle for SDEs driven by $\alpha$-stable process when $\alpha \in (0,1)$.
We also remark that the results in \cite{Vere90, RX21, Xie23}
can deal with the case where the slow process $X_t^\eps$
influences the dynamic of $Y_t^\eps$.

By \cite[Example 2.2]{SXX22} and \cite[Example 1]{LiuD10}, it is observed that
the optimal strong convergence rate for \eqref{0820:01} is $1-\frac{1}{\alpha'}$
when the fast process $Y_t^\eps$ is the solution to SDEs driven by $\alpha'$-stable
process with $\alpha'\in(1,2)$.
In \cite{Xie23}, assuming $\alpha'=2$,
it is shown that the optimal strong convergence rate is $\frac{1}{2}$.
As mentioned in Example \ref{Ex1} below, the optimal rate is $1$ in the case
where the fast process $Y_t^\eps:=\tfrac{t}{\eps}$.
This result is natural since $Y_t^\eps$ is not effected by noise here.
Therefore, our results further demonstrate the effect of the fast process on the convergence rate.

The existing results regarding the strong convergence (with or without rate)
of averaging principle are typically founded on
the classical Khasminskii's time discretization (see e.g. \cite{CHR24}) or
the Poisson equations (see e.g. \cite{RX21, Xie23}).
The Poisson equations are associated with the generator of the fast process,
which requires the diffusion term of fast equation is nondegenerate.
However, the fast equation here covers the case where $\dif Y_t^\eps=\frac{1}{\eps}\dif t$.
To address this question, we use the ODE and the technique of mollification
to capture the fluctuations for the difference between $b(t/\eps,X_t^\eps)$
and $\bar{b}(X_t^\eps)$ (see Section \ref{SCR} for details),
and the strong convergence rate in the averaging principle follows
along with Zvonkin's transform.
To the best of our knowledge, the approach of using ODE
is new in the averaging principle.

Note that the Krylov-type estimate, following from the Schauder estimate,
plays a crucial role in our proof.
Precisely, we first consider the following nonlocal PDEs:
\begin{equation}\label{0412:01}
\partial_tu=\sL_\sigma u+b\cdot\nabla u+f,
\quad u(0)=0,
\end{equation}
where $f\in L_T^\infty(\bC^\vartheta):=L^\infty([0,T];\bC^\vartheta)$ and
for $\phi\in \sS(\R^d)$,
\begin{equation}\label{0414:03}
\sL_\sigma \phi(x):=\int_{\R^d}\left(\phi(x+\sigma(t,x)z)-\phi(x)
-\sigma(t,x)z\b1_{\{\alpha\geq1\}}\b1_{\{|z|\leq1\}}\cdot\nabla\phi(x)\right)\nu(\dif z)
\end{equation}
with a rotationally invariant and symmetric L\'evy measure
$$
\nu(\dif z)=\frac{c_{d,\alpha}}{|z|^{d+\alpha}}\dif z.
$$
Here $c_{d,\alpha}$ is a constant depending on $d,\alpha$.
Assume that $\beta\in(1-\alpha,1)$ when $\alpha\in(0,1]$,
and $\beta\in(\frac{1-\alpha}{2},1)$ when $\alpha\in(1,2)$.
Then we show that the solution $u$ to \eqref{0412:01} satisfies that
\begin{align}\label{0412:02}
\|u\|_{L_T^\infty(B_{\infty,\infty}^{\vartheta+\alpha})}
\leq C\|f\|_{L_T^\infty(\bC^\vartheta)},
\end{align}
where $\vartheta\in((1-\alpha-\beta)\vee(-1),\beta]$ (see Theorem \ref{SchE24}),
and
$B_{p,q}^s$ is the Besov space (see Definition \ref{defBS} for details).
Thanks to \eqref{0412:02}, we establish the following Krylov-type estimate
for any $f\in L^\infty_T(\bC^{\kappa}), \kappa>0$ and $p\geq2$:
\begin{align}\label{0324:01}
\mE\left[\sup_{t\in[0,T]}\left|\int_0^t f(s,X_s^\eps)\dif s\right|^p\right]\leq
C\|f\|_{L^\infty_T(\bC^{-\delta})}^p,
\end{align}
where $\delta\in(0,\alpha/2\wedge(\alpha+\beta-1))\cap[-\beta,\infty)$
(see Lemma \ref{KryEq} for details).

Recall that when $\vartheta=\beta$, the Schauder estimate is established
in \cite{SX23} for $\alpha\in(0,1]$ (also discussed in \cite{CZZ22}).
The case involving $\alpha>1$ and $\beta<0$ is addressed in \cite{LZ22}
(with further insights in \cite{WH23} for more general setting of L\'evy measure).
Additionally, for Schauder's estimates pertaining to a broader spectrum of
L\'evy measures $\nu$, it is referred to \cite{CMP20, Zhao21, Kuhn22, HWW23}.
But to the best of our knowledge, the Schauder estimate presented in Theorem \ref{SchE24}
is new since we can deal with the case where $\vartheta\leq\beta$.
Hence it is interesting on its own rights.
In fact, when $\sigma\equiv c\mI$, we can show that \eqref{SchM0124} holds
for all $\vartheta\in(1-\alpha-\beta,\beta]$ (see Theorem \ref{SchE24}).
In this case, $\vartheta$ can be less than $-1$ and close to $-2$
as $\alpha$ goes to $2$ and $\beta$ approaches to $1$.
We believe that it is also valuable for addressing numerous other questions.


Given $(\alpha,\beta)\in\cA_2$, it is important to note that in this case,
$b$ can be treated as a distribution.
Singular SDEs with distributional drifts arise from many stochastic models;
for further details, refer to \cite{KP22}.
Despite considerable advances in the averaging principle for SDEs, it seems that
there is no work on SDEs with distributional drifts.
Consequently, we also aim to prove the averaging principle
for SDEs with distributional drifts in this paper.
A difficulty that we encounter is how to address the solution
in the context of a singular drift.
In this regard, we adopt the martingale solution approach
as defined by Either and Kurtz \cite{EK86}
(see Definition \ref{MSolution} for details).
Then with the help of mollification method and Zvonkin's transform
we obtain that the martingale solution
of \eqref{Main01} converges weakly to that of the averaged equation \eqref{0318:03}.


To provide a clear summary of our results on the strong and weak
convergence of the averaging principle to the SDE \eqref{Main01} driven by an
$\alpha$-stable process with a $\beta$-H\"older drift,
we give the range of $(\alpha,\beta)$ in
Figure $1$ below.
When $(\alpha,\beta)\in\cA_1$, we obtain the strong convergence of the
averaging principle.
In particular, we have the optimal strong convergence rate provided
$(\alpha,\beta)\in\cA_0\subset\cA_1$.
Moreover, if $(\alpha,\beta)\in\cA_2$, we establish the weak
convergence of the averaging principle.
\begin{figure}[H]
\centering
\begin{tikzpicture}[scale=2.85]
\draw [step=0.05cm, gray, very thin, dashed] (0,-0.8) grid (2.4,1.2);
\draw [thin, ->] (0,0)--(2.5,0) node[anchor=north west] {$\alpha$};
\draw [thin, ->] (0,-0.8)--(0,1.3) node[anchor=south east] {$\beta$};
\foreach \x/\xtext in {0.5/\tfrac{1}{2}, 1, 1.5/\tfrac{3}{2}, 2}
\draw (\x cm, 1pt) -- (\x cm, 0pt) node[anchor=north] {$\xtext$};
\foreach \y/\ytext in {-0.5/-\tfrac{1}{2}, 0, 0.5/\tfrac{1}{2}, 1}
\draw (1pt, \y cm) -- (0pt, \y cm) node[anchor=east] {$\ytext$};
\draw [thick, fill=green,opacity=0.2] (0,1)--(2,1) --(2,0)--(0,1);
\node at (0.6,0.8) {$\cA_1$};
\node at (1.6,0.5) {$\cA_1$};
\draw [thick, fill=blue,opacity=0.2] (2/3,1)--(1,0.5) --(2,1)--(2/3,1);
\node at (1.2,0.8) {$\cA_0$};
\draw [thick, fill=pink, opacity=0.2] (0,1)--(1,0)--(2,-0.5) --(2,0)--(0,1);
\node at (1,0.2) {$\cA_2$};
\draw [fill] (2/3,1) circle [radius=0.005];
\node [above] at (2/3,1) {$(\tfrac{2}{3},1)$};
\draw [fill] (1,0.5) circle [radius=0.005];
\node [below] at (1,0.5) {$(1,\tfrac{1}{2})$};
\draw [fill] (2,1) circle [radius=0.005];
\node [right] at (2,1) {$(2,1)$};
\draw [fill] (2,-0.5) circle [radius=0.005];
\node [below] at (2,-0.5) {$(2,-\tfrac{1}{2})$};
\end{tikzpicture}
\caption{The domain of $\alpha$ and $\beta$.}
\end{figure}
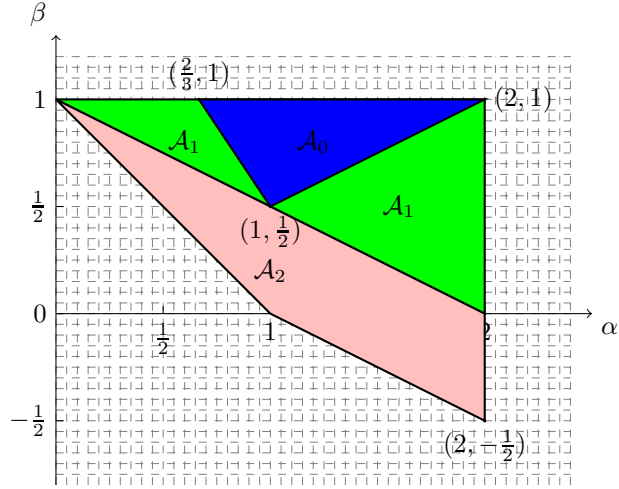

Now we conclude this section by summarizing the structure of the paper.
In Section \ref{MR}, we state our main results.
In Section \ref{Pre}, we recall some definitions and facts
concerning Besov space and $\alpha$-stable process.
In Section \ref{SE}, we prove the Schauder estimate for the
nonlocal PDE \eqref{Pareq24} and the Krylov-type estimate.
Finally, we prove the strong and weak convergence for
the averaging principle in Section \ref{AP}.

~~\\

{\bf Notations.}

Throughout this paper, we use $:=$ to mean a way of definition.
Define $a\vee b:=\max\{a,b\}$ and $a\wedge b:=\min\{a,b\}$ for $a,b\in\R$.
Write $\nabla:=\left(\frac{\partial}{\partial x_1},...,\frac{\partial}{\partial x_d}\right)$
and $\Delta:=\sum_{k=1}^d\frac{\partial^2}{\partial x_k^2}$ on $\R^d$.
We use the letter $C$ to denote some constant,
which may change from line to line.
Let $C(\theta)$ indicate that $C$ depends on
$\theta$.
By $A\lesssim B$ we mean that $A\leq CB$ for some constant $C\geq1$,
and the notation $A\lesssim_\theta B$ indicates that $A\leq C(\theta)B$.
Let $\supp[\varphi]:=\overline{\{x\in\R^d:\varphi(x)\neq0\}}$
for $\varphi:\R^d\rightarrow\R$.
For any $A\subset\R^d$ define the indicator function $\b1_{A}$,
which means that $\b1_{A}(x)=1$ for $x\in A$ and $\b1_{A}(x)=0$ otherwise.
Let $\mI$ be the identity mapping on $\R^d$. For two operators $T_1$ and $T_2$, we use $[T_1,T_2]:=T_1T_2-T_2T_1$ to denote the commutator.

\section{Statement of main results}\label{MR}

In this section, we formulate our main results.
As mentioned in the introduction, the Krylov-type estimate plays a crucial role
in the proof of the averaging principle, with its derivation stemming from Schauder's estimate.
In the sequel, we fix a time $T>0$. We first consider the following nonlocal PDEs:
\begin{equation}\label{Pareq24}
\partial_tu_\lambda=\sL_\sigma u_\lambda-\lambda u_\lambda+g\cdot\nabla u_\lambda+f,
\quad u_\lambda(0)=0,
\end{equation}
where $\lambda\geq0$, $f\in L_T^\infty(\bC^\vartheta)$, $g\in (L_T^\infty(\bC^\beta))^d$ and
$\sL_\sigma$ is as in \eqref{0414:03}.
Here $\vartheta,\beta$ are specified below, and $\sigma$ satisfies the following condition:
\begin{itemize}
  \item[{\bf(H$_{\sigma}$)}]
  There exists $\Lambda_1>1$ such that for all $t\geq0$, $x,\xi\in\R^d$,
  \begin{equation*}
    \Lambda_1^{-1}|\xi|^2\leq |\sigma(t,x)\xi|^2\leq\Lambda_1|\xi|^2,\quad
    |\nabla_x\sigma(t,x)|\leq\Lambda_1.
  \end{equation*}
\end{itemize}
In the following we  use the following parameter sets
$$
\Theta:=(\alpha,d,T,\beta,\vartheta, \Lambda_1).
$$

The following theorem is our first main result,
which provides a useful Schauder's estimate for \eqref{Pareq24}.
This theorem serves as a foundational element in our subsequent analysis and proofs.
\begin{theorem}\label{SchE24}
Let $\beta\in(1-\alpha,1)$ for $\alpha\in(0,1]$
and $\beta\in(\frac{1-\alpha}{2},1)$ for $\alpha\in(1,2)$.
Let
$$
\vartheta\in((1-\alpha-\beta)\vee(-1),\beta]\mbox{ or if $\sigma=c\mI$}, \ \vartheta\in(1-\alpha-\beta,\beta].
$$
Under {\bf (H$_\sigma$)},  for each $\lambda\geq0$, $f\in L_T^\infty(\bC^\vartheta)$
and $g\in (L_T^\infty(\bC^\beta))^d$,
there exits a unique solution
$u_\lambda$ to \eqref{Pareq24} with that
\begin{align}\label{SchM0124}
\|u_\lambda\|_{L_T^\infty(B_{\infty,\infty}^{\vartheta+\eta})}
\lesssim_C(1+\lambda)^{-\frac{\alpha-\eta}{\alpha}} \|f\|_{L_T^\infty(\bC^\vartheta)},\ \forall \eta\in[0,\alpha],
\end{align}
where $C=C(\Theta,\eta,\|g\|_{\mL^\infty_T(\bC^\beta)})>0$.
\end{theorem}

\begin{remark}\rm
Compared with the existing results \cite{CZZ22, LZ22, SX23},
we allow $\vartheta\leq\beta$ in Theorem \ref{SchE24}.
\end{remark}

Let $\Xi$ be the set of all mappings $\ell:\R_+\rightarrow\R_+$ satisfying that
$t\mapsto t\ell(t)$ is non-decreasing and $\lim_{t\to\infty}\ell(t)=0$.
Now we introduce the following averaged conditions about coefficients
$b$ and $\sigma$.
\begin{itemize}
  \item [{\bf (A$_b$)}] Let $b\in L^\infty(\R_+;\bC^\beta)$, where $\beta\in\mR$.
  There exist $\gamma\leq\beta$, $\bar{b}\in\bC^\beta$ and $\ell_1\in\Xi$
       such that
  \begin{align}\label{B00}
\left\|\frac{1}{T}\int_{0}^{T}b(s,\cdot)\dif s -\bar{b}(\cdot)\right\|_{\bC^\gamma}\le \ell_1(T).
\end{align}
  \item [{\bf (A$_\sigma$)}] There exist $\ell_2\in\Xi$ and $\bar{\sigma}$ satisfying
  {\bf (H$_\sigma$)} such that
  \begin{equation*}
  \left\|\frac{1}{T}\int_0^{T}\left|\sigma(s,\cdot)-\bar{\sigma}(\cdot)\right|^2\dif s\right\|_{\bC^1}
  \leq\ell_2(T).
  \end{equation*}
\end{itemize}

Let $(\alpha,\beta)\in\cA_1$.
By Theorem \ref{SchE24}, we establish the Krylov-type estimate \eqref{0324:01}.
Due to the low regularity of the coefficients, we need to use mollification techniques.
Therefore, assuming $b\in L_T^\infty(\bC_b^\infty)$, via an ODE,
we can capture the fluctuations as follows:
\begin{align}\label{0414:01}
&
\mE\left(\sup_{t\in[0,T]}\left|\int_0^t
\left(b\left(\frac{s}{\eps},X_s^\eps\right)-\bar{b}\left(X_s^\eps\right)\right)\dif s\right|^p\right)\\\nonumber
&
\lesssim_{T}\eps^p\sup_{t\in[0,T]}\left\|\int_0^{\frac{t}{\eps}}\left(b(s,\cdot)-\bar{b}(\cdot)\right)
\dif s\right\|_{\bC^\kappa}^p
\lesssim_T \left(\ell_1(T/\eps)\right)^p
\end{align}
for any $0<\eps\ll1$, $p\geq1$ and $\kappa>(1-\alpha/2)\vee\alpha/2$ (see Lemma \ref{0116Lem}).
Similarly, we have the fluctuations between $\sigma$ and $\bar{\sigma}$.
Thus we deduce the following strong convergence rate for the averaging principle
in the case where $(\alpha,\beta)\in\cA_1$.

\bt\label{Th1}
Assume that {\bf (A$_b$)} and  one of the following conditions hold:
\begin{itemize}
  \item[(i)] $\alpha\in(0,1]$, $p\geq1$  and $\sigma=c\mI$, where $c>0$;
  \item[(ii)] $\alpha\in(1,2)$, $p\in[1,\alpha)$, {\bf (H$_\sigma$)} and {\bf (A$_\sigma$)} hold.
\end{itemize}
For any $\beta\in(1-\tfrac{\alpha}{2},1)$, there exists a constant $C>0$
such that for any $\eps>0$,
\begin{align}\label{Rate:01}
\mE\left(\sup_{t\in[0,T]}|X^\eps_t-\bar{X}_t|^p\right)
\le C \left[\left(\ell_1\left(\tfrac{T}{\eps}\right)\right)^{\frac{p(1\wedge \alpha)}{(1\wedge \alpha)+\delta_1}}
+\left(\ell_2\left(\tfrac{T}{\eps}\right)\right)^{\frac{p}{2}}\right],
\end{align}
where $\delta_1:=[((1-\tfrac{\alpha}{2})\vee\tfrac{\alpha}{2}
-\gamma)\vee(2-\tfrac{3\alpha}{2}-\beta)+\iota]\vee0$ with
any $\iota>0$.
\et

We show via the following simple example that $\ell_1(\tfrac{T}{\eps})$
is the optimal strong convergence rate
when $\sigma$ is time-independent.
\begin{Examples}\rm\label{Ex1}
Consider the following SDE:
\begin{equation*}
\dif X_t^\eps=\cos\left(t/\eps\right)\dif t+\dif L_t^{(\alpha)}, \quad X_0^\eps=x\in\R.
\end{equation*}
It is obvious that the averaged equation is
\begin{equation*}
\dif \bar{X}_t=\dif L_t^{(\alpha)}, \quad \bar{X}_0=x\in\R
\end{equation*}
with $\ell_1(T)=1/T$, and for $\eps<T/\pi$,
\begin{equation*}
\sup_{t\in[0,T]}|X_t^\eps-\bar{X}_t|
=\sup_{t\in[0,T]}\left|\eps\sin(t/\eps)\right|=\eps.
\end{equation*}
\end{Examples}

\begin{remark}\rm\label{Rk1}
(1) In many examples, $\ell_1(1/\eps)\sim \eps$, such as $b$ is periodic.

(2) In order to clearly demonstrate the convergence rate,
we can assume that $\sigma$ is time-independent. Then we will have the following conclusions.
\begin{itemize}
  \item[(i)] If $b$ can be expressed in the form of $b(t,x)=b_0(t)b_1(x)$
for all $(t,x)\in\R_+\times\R^d$,
then we have $\gamma=\beta$ and $\delta_1=\left[\left(\tfrac{\alpha}{2}-\beta\right)
\vee\left(2-\tfrac{3\alpha}{2}-\beta\right)
+\iota\right]\vee0$.
Thus, we obtain the optimal strong convergence rate $\ell_1(\tfrac{T}{\eps})$ provided $(\alpha,\beta)\in\cA_0$.

\item[(ii)]
If $\gamma=0$ then the convergence rate is
  \begin{equation*}
  r_1(\alpha)=
  \begin{cases}
  \frac{2\alpha}{2+\alpha}-, & \alpha\in(0,1)\\
  \frac{2}{3}-, & \alpha=1\\
  \frac{2}{2+\alpha}-, & \alpha\in(1,2)
  \end{cases}
  \end{equation*}
where $a-:=a-\iota$ with some $\iota>0$. In this case, the convergence rate is independent with
the regularity of the drift $b$.
Specifically, when $\alpha=1$, the rate $r_1(\alpha)$ tends to the supremum
$\frac{2}{3}$ for any $\beta\in(\tfrac12,1)$.
Furthermore, as $\alpha$ goes to $2$, for any $\beta\in(0,1)$, $r_1(\alpha)$ tends to $\frac{1}{2}-$,
which is faster than $\frac{1}{3}-$ in \cite{CHR24}.

\item[(iii)] We rewrite $\delta_1$ as follows:
\begin{equation*}
\delta_1=
\begin{cases}
\left[\left(1-\frac{\alpha}{2}-\gamma\right)\vee\left(2-\frac{3\alpha}{2}-\beta\right)+\iota\right]
\vee0,
& \alpha\in(0,1],\\
\left[\left(\frac{\alpha}{2}-\gamma\right)\vee(2-\frac{3\alpha}{2}-\beta)+\iota\right]\vee0,
& \alpha\in(1,2).
\end{cases}
\end{equation*}
Consider $(\alpha,\beta)\in\cA_1\setminus\cA_0$.
Since $\gamma\leq\beta$, we obtain that $\delta_1$ reaches its minimum when $\alpha=1$.
Consequently, this implies that the rate goes to its maximum value when $\alpha=1$.
\end{itemize}

(3) The generator of $X_t^\eps,t\geq0$ is given by
\begin{equation*}
\sL_\eps:=b\left(\tfrac{t}{\eps},\cdot\right)\cdot\nabla +\sL_{\sigma(t/\eps,\cdot)},
\end{equation*}
where $\sL_{\sigma(t/\eps,\cdot)}$ is defined as in \eqref{0414:03}
but with $\sigma(t,\cdot)$ replaced by $\sigma(t/\eps,\cdot)$.
The reason why $r_1(\alpha)$ reaches its maximum
at $\alpha=1$ for a given $\beta$ is that we use the fluctuations \eqref{0414:01}
and employ the technique of mollification to derive the strong convergence rate.
The smaller the $\kappa$, the faster the rate (see \eqref{0403:06} for details).
The lower bounds $1-\alpha/2$ and $\alpha/2$ of $\kappa$ come from the first order
term $b\left(\tfrac{t}{\eps},\cdot\right)\cdot\nabla$ and $\sL_{\sigma(t/\eps,\cdot)}$ respectively.
When $\alpha>1$, $\sL_{\sigma(t/\eps,\cdot)}$ is the dominant
term. However, when $\alpha\in(0,1)$, the gradient term $b\left(\tfrac{t}{\eps},\cdot\right)\cdot\nabla$
is of higher order than $\sL_{\sigma(t/\eps,\cdot)}$ (see Lemma \ref{0116Lem}).
\end{remark}

Actually, our result Theorem \ref{Th1}
can be applied to the following slow-fast systems:
\begin{equation}\label{0815:02}
\dif X_t^\eps=f(X_t^\eps,Y_t^\eps)\dif t+\sigma(X_t^\eps)\dif L_t^{(\alpha)},
\quad X_0^\eps=x\in\R^d,
\end{equation}
where $f:\R^{n}\times\R^m\rightarrow\R^n$ belongs to $L^\infty(\R^m;\bC^\beta(\R^n))$ and
the fast process $Y_t^\eps:=Y_{t/\eps}$, defined on some probability space
$(\Omega',\cF',\mP')$, is independent of $L_t^{(\alpha)}$.
In this case, we need the following condition {\bf(A$_f$)}, which can be viewed as a kind of
strong law of large numbers:
\begin{itemize}
  \item [{\bf(A$_f$)}]
There exist $\bar{f}\in\bC^\beta(\R^n)$, $\gamma\leq\beta$ and
$\ell:\R_+\times\Omega'\rightarrow \R_+$ with $\ell(\cdot,\omega')\in\Xi,$ $\mP'$-a.s.
such that
\begin{equation*}
\left\|\frac{1}{T}\int_0^T
f(\cdot,Y_s(\omega'))\dif s-\bar{f}(\cdot)\right\|_{\bC^\gamma}
\leq \ell(T,\omega')
\quad \mP'-{\rm a.s.}
\end{equation*}
\end{itemize}

\begin{remark}\rm
Note that condition {\bf(A$_f$)} is not difficult to verify.
Let $Y_t^y:=Y_t$ with $Y_0=y\in\R^m$.
Assume that $Y_t^y$ is Markovian.
Define $P(t,y,A):=\mP'(Y_t^y\in A)$, $A\in\cB(\R^m)$. If $P$ admits an
invariant measure $\mu$ such that
\begin{equation*}
\lim_{t\rightarrow\infty}\|P(t,y,\cdot)-\mu\|_{TV}=0 \quad {\text{for every}}
~y\in\R^m,
\end{equation*}
then for any $\mu$-integrable $g$ and $y\in\R^m$,
\begin{equation*}
\lim_{T\rightarrow\infty}\frac{1}{T}\int_0^Tg(Y_t^y)\dif t
=\int_{\R^m}g(y)\mu(\dif y)
\quad \mP'-{\text{a.s.}};
\end{equation*}
see e.g. (2.24) in \cite{Bhat82}.
Let $f(x,y):=\varphi(x)g(y)$. If $\varphi\in\bC^{\beta}$ then
$f$ satisfies {\bf(A$_f$)} with $\gamma=\beta$.
\end{remark}

Then with the help of Theorem \ref{Th1},
we have the following result.
\begin{theorem}\label{thSF}
Assume that {\bf(A$_f$)} and one of the following conditions hold:
\begin{itemize}
  \item[(i)] $\alpha\in(0,1]$, $p\geq1$  and $\sigma=c\mI$, where $c>0$;
  \item[(ii)] $\alpha\in(1,2)$, $p\in[1,\alpha)$ and
  {\bf (H$_\sigma$)} holds.
\end{itemize}
Then for any $\beta\in(1-\tfrac{\alpha}{2},1)$, there exists a constant $C>0$ such that
\begin{equation}
\mE\left(\sup_{t\in[0,T]}\left|X_t^\eps-\bar{X}_t\right|^p\right)
\leq C \mE'\left([\ell(\tfrac{T}{\eps})]^{\frac{p(1\wedge\alpha)}{(1\wedge\alpha)+\delta_1}}\right),
\end{equation}
where $\delta_1$ is as in Theorem \ref{Th1} and $\bar{X}_t$ satisfies
\begin{equation*}
\dif \bar{X}_t=\bar{f}(\bar{X}_t)\dif t+\sigma(\bar{X}_t)\dif L_t^{(\alpha)},
\quad \bar{X}_0=x.
\end{equation*}
\end{theorem}
\begin{proof}
Define
$$
b(t,x)(\omega'):=f(x,Y_t(\omega')).
$$
Then we can rewrite \eqref{0815:02} as follows
\begin{equation*}
\dif X_t^\eps=b(t/\eps,X_t^\eps)(\omega')\dif t+\dif L_t^{(\alpha)}.
\end{equation*}
By Theorem \ref{Th1}, one sees that
\begin{equation}\label{0815:03}
\mE\left(\sup_{t\in[0,T]}|X_t^\eps(\cdot,\omega')-\bar{X}_t(\cdot)|^p\right)
\lesssim [\ell(T/\eps)(\omega')]^{\frac{p(1\wedge\alpha)}{1\wedge\alpha+\delta}}.
\end{equation}
Taking expectations on both side of \eqref{0815:03}, we complete the proof.
\end{proof}

\begin{Examples}\rm\label{Ex3}
Consider the following slow-fast system:
\begin{equation}\label{0819:01}
\begin{cases}
\dif X_t^\eps=f(X_t^\eps,Y_t^\eps)\dif t+\sigma(X_t^\eps)\dif L_t^{(\alpha)},\quad X_0^\eps=x_0\in\R^n\\
\dif Y_t^\eps=\frac{1}{\eps}B(Y_t^\eps)\dif t+\frac{1}{\sqrt{\eps}}\dif W_t,\quad Y_0^\eps=y_0\in\R^m,
\end{cases}
\end{equation}
where $f\in L^\infty(\R_y^m;\bC_b^{\beta_1})\cap L^\infty(\R_x^n;\bC_b^{\beta_2})$
with $(\beta_1,\beta_2)\in(1-\tfrac{\alpha}{2},1)\times(0,1]$,
$\sigma$ satisfies {\bf (H$_\sigma$)}  and
$B$ satisfies that
\begin{equation*}
\lim_{|y|\rightarrow\infty}\langle B(y),y\rangle=-\infty.
\end{equation*}
Let $Y_t^{y_0}$ be the solution to
\begin{equation}\label{0822:01}
\dif Y_t=B(Y_t)\dif t+\dif \widetilde{W}_t,\quad Y_0=y_0,
\end{equation}
where $\widetilde{W}_t:=\tfrac{1}{\sqrt{\eps}}W_{\eps t}$.
Note that $Y_t^\eps=Y_{t/\eps}^{y_0}$.
It follows from \cite[Theorem 2]{Vere97} that there exists a unique invariant
measure $\mu$ such that
for any $k>0$, $m>2k+2$ and $\varphi\in L^\infty$,
\begin{equation}\label{0819:02}
\left|\mE\varphi(Y_t^{y_0})-\mu(\varphi)\right|\leq C (1+|y_0|^m)(1+t)^{-(k+1)}\|\varphi\|_\infty.
\end{equation}
By \cite[Lemma 1]{PV03}, we have for any $p\geq1$,
\begin{equation*}
\mE\left(|Y_t^{y_0}|^{p}\right)\leq C_p(1+|y_0|^{p+2}).
\end{equation*}
Hence, by \cite[Corollary 2.4]{Shi06}, one sees that for any $\kappa\in(0,1/2)$
there is a random variable $\xi_{\kappa,f}\geq1$ such that
\begin{equation}\label{0820:04}
\left|\frac{1}{T}\int_0^Tf(x,Y_t^{y_0})\dif t-\int f(x,y)\mu(\dif y)\right|
\leq \xi_{\kappa,f}\|f(x,\cdot)\|_\infty T^{-\frac{1}{2}+\kappa},
\end{equation}
where $\xi_{\kappa,f}$ satisfies that for any $m\geq1$,
\begin{equation}\label{0820:05}
\mE\xi_{\kappa,f}^m\leq C_{m,k}\|f(x,\cdot)\|^{m(m+1)}_{\bC^{\beta_2}}(|y_0|^{p+2}+1).
\end{equation}
In particular, letting $f_{x_1,x_2}(\cdot):=f(x_1,\cdot)-f(x_2,\cdot)$,
we have
\begin{align}\label{0820:06}
\left|\frac{1}{T}\int_0^Tf_{x_1,x_2}(Y_t^{y_0})\dif t-\int f_{x_1,x_2}(y)\mu(\dif y)\right|
\leq \xi_{\kappa,f_{x_1,x_2}}\|f_{x_1,x_2}\|_\infty T^{-\frac{1}{2}+\kappa}.
\end{align}

It follows from \eqref{0820:04}, \eqref{0820:05} and \eqref{0820:06} that
{\bf(A$_f$)} holds with $\gamma=\beta=\beta_1$ and
$$
\bar{f}(x):=\int_{\R^m}f(x,y)\mu(\dif y).
$$
Hence, if $\sigma=c\mI$ with $c>0$ when $\alpha\in(0,1]$,
then by Theorem \ref{thSF} we obtain that
\begin{equation*}
\mE\left(\sup_{t\in[0,T]}\left|X_t^\eps-\bar{X}_t\right|\right)
\lesssim \eps^{(\frac{1}{2}-\kappa)\frac{1\wedge\alpha}{1\wedge\alpha+\delta_1}}.
\end{equation*}
Define the order of convergence rate by
$$
R:=\frac{(\frac12-\kappa)(1\wedge\alpha)}{1\wedge\alpha+\delta_1}.
$$
If $(\alpha,\beta)\in\cA_0$ then $R=\tfrac12-$.
Recall that
$\delta_1=\left[\left(\tfrac{\alpha}{2}-\beta\right)
\vee\left(2-\tfrac{3\alpha}{2}-\beta\right)
+\iota\right]\vee0$.
Since $\iota>0$ is any small constant,
letting $\iota\in(0,\beta_1-1+\tfrac{\alpha}{2})$, we have
\begin{align*}
R
&
=\frac{(\frac12-\kappa)(1\wedge\alpha)}{1\wedge\alpha
+[(\frac{\alpha}{2}-\beta_1)\vee(2-\frac{3\alpha}{2}-\beta_1)+\iota]\vee0}\\
&
>\frac{(\frac12-\kappa)(1\wedge\alpha)}{1\wedge\alpha
+[(\alpha-1)\vee(1-\alpha)]}=(\frac12-\kappa)(\alpha\wedge\frac{1}{\alpha}).
\end{align*}

\end{Examples}

Now we introduce the following result concerning the weak convergence of
the averaging principle for SDEs when $(\alpha,\beta)\in\cA_2$,
which includes distributional drifts.
\begin{theorem}\label{Th2}
Assume that {\bf (A$_b$)} and  one of the following conditions hold:
\begin{itemize}
  \item[(i)] $(\alpha,\beta)\in(0,1]\times(1-\alpha,1-\tfrac{\alpha}{2}]$  and $\sigma=c\mI$, where $c>0$;
  \item[(ii)] $(\alpha,\beta)\in(1,2)\times(\tfrac{1-\alpha}{2},1-\tfrac{\alpha}{2}]$,
  {\bf (H$_\sigma$)} and {\bf (A$_\sigma$)} hold.
\end{itemize}
For any $x\in\R^d$ and $t\geq 0$, let $\mP_x^\eps(t)$ and $\mP_x(t)$ be the time marginal laws of the martingale solutions to \eqref{Main01}  and  \eqref{0318:03}, respectively.
We have
\begin{align}\label{Rate:02}
\lim_{\eps\to0}\sup_{t\in[0,T]}\sup_{x\in\R^d}
\cW_1\left(\mP_x^\eps(t),\mP_x(t)\right)=0,
\end{align}
where for $\mu,\nu\in\cP(\R^d)$,
$\cW_1(\mu,\nu)$ is the Kantorovich-Rubinstein metric defined by
$$
\cW_1(\mu,\nu):=\sup\left\{\left|\int_{\R^d}f(\dif\mu-\dif\nu)\right|
:~\|f\|_\infty+\mathop{\sup}_{x,y\in\R^d\atop x\neq y}\frac{|f(x)-f(y)|}{|x-y|}\leq1\right\}.
$$
\end{theorem}

To illustrate our results, we introduce the following example.
\begin{Examples}\rm
Let $\{Z_t^z,t\geq0\}_{z\in\R^{m}}$ be a family of Markov processes.
Let $P_t\varphi(z):=\mE \varphi(Z_t^z)$ be the associated semigroup. Suppose that
$(P_t)_{t\geq 0}$ admits a unique stationary measure $\mu$ such that
\begin{equation}\label{0715:01}
\sup_{z\in\mR^d}\sup_{\|\varphi\|_\infty\leq 1}
\left|\frac{1}{T}\int_0^T P_t\varphi(z)\dif t
-\mu(\varphi)\right|\leq \ell(T),
\end{equation}
where $\ell\in\Xi$ and $\mu(\varphi):=\int_{\R^d}\varphi (z)\mu(\dif z)$.
Suppose that $\varphi(z,x):\R^n\times\mR^n\rightarrow\R^{n}$ satisfies
$$
\varphi\in L^\infty(\R^n;\bC^\beta(\R^n;\mR^n)) \quad {\text{with}}
\quad \beta\in[0,1).
$$
Consider  the following perturbation equation:
\begin{equation*}
\dif X_t^\eps=[P_{t/\eps}\varphi(\cdot,X_t^\eps)](X_t^\eps)\dif t+\dif L_t^{(\alpha)}, \quad X_0^\eps=x\in\R^n,
\end{equation*}
and the averaged equation
\begin{equation*}
\dif \bar{X}_t=\mu(\varphi(\cdot,\bar{X}_t))\dif t+\dif L_t^{(\alpha)},\ \ \bar X_0=x.
\end{equation*}
By \eqref{0715:01}, it is easy to see that {\bf (A$_b$)} holds for
$b(t,x):=P_{t}\varphi(\cdot,x)(x)$ and $\bar{b}(x):=\mu(\varphi(\cdot,x))$.
Now by Theorems \ref{Th1} and \ref{Th2}, we have the following conclusions:
\begin{enumerate}[(i)]
\item
If $(\alpha,\beta)\in\cA_1$, then it follows from Theorem \ref{Th1}, \eqref{0715:01}
and Remark \ref{Rk1} (2)(ii)
that there exists a constant $C>0$ such that for any $0<\eps\ll1$,
\begin{equation*}
\mE\left(\sup_{t\in[0,T]}|X_t^\eps-\bar{X}_t|\right)
\leq C\ell(\tfrac{T}{\eps})^{\frac{2\wedge\alpha}{2+\alpha}-}.
\end{equation*}

\item If $(\alpha,\beta)\in\cA_2$ and $\beta\geq0$, then it follows from Theorem \ref{Th2} that
\begin{equation*}
\lim_{\eps\rightarrow0}\sup_{t\in[0,T]}
\cW_1\left(\cL_{X_t^\eps},\cL_{\bar{X}_t}\right)=0.
\end{equation*}
\end{enumerate}
\end{Examples}

\section{Preliminaries}\label{Pre}
In this section, we introduce some definitions and results concerning Besov spaces
and $\alpha$-stable process.

\subsection{Besov spaces}\label{BS}
Let $\sS(\R^d)$ be the Schwartz space of all rapidly decreasing functions,
and $\sS'$ its dual, which is called Schwartz generalized function
(or tempered distribution) space.
For $f\in\sS(\R^d)$, define its Fourier transform and inverse transform
as follows:
$$
\cF f(\xi)=\hat{f}(\xi):=(2\pi)^{-d/2}\int_{\R^d}\e^{-i\xi\cdot x}f(x)\dif x,\quad
\cF^{-1}f(\xi)=\check{f}(\xi):=(2\pi)^{-d/2}\int_{\R^d}\e^{i\xi\cdot x}f(x)\dif x
$$
For $f\in\sS'(\R^d)$, by Schwartz's duality,
$\hat{f}$ and $\check{f}$ are the unique elements in $\sS'(\R^d)$
so that (see \cite{BCD11})
$$
\langle \hat{f},\phi\rangle=\langle f,\hat{\phi}\rangle,
\quad
\langle \check{f},\phi\rangle=\langle f,\check{\phi}\rangle,
\quad
\forall \phi\in\sS(\R^d).
$$
Let $\chi:\R^d\rightarrow[0,1]$ be a smooth radial function such that
$\chi(\xi)=1$ for $|\xi|\leq 1$ and $\chi(\xi)=0$ for $|\xi|\geq 3/2$.
Define
$$
\varphi(\xi):=\chi(\xi)-\chi(2\xi).
$$
It is easy to see that $\varphi\geq0$ and for each $\xi\in\R^d$,
\begin{equation}
\chi(2\xi)+\sum_{j=0}^k\varphi(2^{-j}\xi)=\chi(2^{-k}\xi)\rightarrow1,
\quad
{\text{as}} ~k\rightarrow\infty,
\end{equation}
as well as
\begin{equation*}
{\text{supp}}[\varphi(2^{-j}\cdot)]\cap{\text{supp}}[\varphi(2^{-k}\cdot)]=\emptyset
\quad {\text{if}}~|j-k|\geq2.
\end{equation*}
From now on, we fix such $\chi$ and $\varphi$, and define the following dyadic operators
$\cR_j$ on $\sS'(\R^d)$:
\begin{equation*}
\cR_jf:=
\begin{cases}
\cF^{-1}\left(\chi(2\cdot)\cF f\right), & j=-1\\
\cF^{-1}\left(\varphi(2^{-j}\cdot)\cF f\right), & j\geq0.
\end{cases}
\end{equation*}
Setting $h_j:=\cF^{-1}\varphi(2^{-j}\cdot)$ for $j\geq0$ and $h_{-1}:=\cF^{-1}\chi(2\cdot)$,
we have $\cR_jf(x)=\int_{\R^d}h_j(y)f(x-y)\dif y$,
and $h_j(\cdot)=2^{jd}h_0(2^j\cdot)$, which implies that
for all $\theta>0$ and $k\in\N$,
\begin{align}\label{0301:01}
\int_{\R^d}|y|^\theta|\nabla^kh_j(y)|\dif y
=2^{(k-\theta)j}\int_{\R^d}|y|^\theta|\nabla^kh_0(y)|\dif y
\lesssim 2^{(k-\theta)j}.
\end{align}

Now we can introduce the definition of Besov spaces.
\begin{definition}\label{defBS}
For any $s\in\R$ and $p,q\in[1,\infty]$, the Besov space $B_{p,q}^s$ is defined as the set of all
$f\in\sS'(\R^d)$ with
\begin{equation*}
\|f\|_{B_{p,q}^s}:=\b1_{\{q<\infty\}}\left(\sum_{j\geq-1}2^{jsq}\|\cR_jf\|_{L^p}^q\right)^{1/q}
+\b1_{\{q=\infty\}}\left(\sup_{j\geq-1}2^{js}\|\cR_jf\|_{L^p}\right)<\infty.
\end{equation*}
\end{definition}
\begin{remark}\rm
(i) For $s\in\R$, let $\bC^s$ be the usual H\"older space when $s>0$,
and define $\bC^{s}:=B_{\infty,\infty}^s$ for $s\leq0$.
Note that if $s>0$ and $s\notin\N$, then there exist constants $C_1,C_2>0$ such that
\begin{equation*}
C_1\|f\|_{\bC^s}\leq \|f\|_{B_{\infty,\infty}^s} \leq C_2 \|f\|_{\bC^s}.
\end{equation*}
In the case where $s\in\N$, the space $B_{\infty,\infty}^s$ is strictly larger
than $\bC^s$, and
there exists a constant $C>0$ such that
\begin{equation*}
\|f\|_{B_{\infty,\infty}^s}\leq C\|f\|_{\bC^s}.
\end{equation*}

(ii) Define $\cR_{-2}:=0$. Note that for all $j\geq-1$ we have
\begin{equation*}
\int_{\R^d}\cR_jf(x)g(x)\dif x=\int_{\R^d}f(x)\cR_jg(x)\dif x,
\quad \forall f\in\sS'(\R^d), g\in\sS(\R^d),
\end{equation*}
and
\begin{equation}\label{0123:05}
\cR_j=\cR_j\widetilde{\cR}_j,
\end{equation}
where $\widetilde{\cR}_j:=\cR_{j-1}+\cR_{j}+\cR_{j+1}$.
\end{remark}

Recall the following Bernstein's inequality; see e.g. \cite[Lemma 2.2]{CZZ22}.
\begin{lemma}\label{Lem28}
Let $1\leq p\leq q\leq\infty$. For any $k=0,1,2...$ and $\gamma\in(-1,2)$,
there exits a constant $C=C(k,p,q,d)>0$
such that for all $f\in\sS'(\R^d)$ and $j\geq-1$,
\begin{equation}\label{Berneq01}
\|\nabla^k\cR_jf\|_q\lesssim_C2^{(k+d(\frac{1}{p}-\frac{1}{q}))j}\|\cR_jf\|_p,
\end{equation}
and for any $j\geq0$,
\begin{equation}\label{Serneq02}
\|\left(-\Delta\right)^{\frac\gamma2}\cR_jf\|_q
\lesssim_C2^{\left(\gamma+d(\frac1p-\frac1q)\right)j}\|\cR_jf\|_p.
\end{equation}
\end{lemma}

\br\rm
In view of Lemma \ref{Lem28}, we have for any $1\le p\le q\le\infty$,
$r\in[1,\infty]$, $s\in\mR$ and $\eps>0$, the following embedding is continuous
\begin{align}\label{1128:00}
B^{s+d(\frac1p-\frac1q)}_{p,r}\hookrightarrow B^{s}_{q,r} \hookrightarrow B^{s}_{q,\infty} \hookrightarrow B^{s-\eps}_{q,1}.
\end{align}
\er

The following commutator estimate is proven in \cite[Lemma 2.3]{CZZ22}.
\begin{lemma}
Let $p,p_1,p_2,q_1,q_2\in[1,\infty]$ satisfying $\frac{1}{p}=\frac{1}{p_1}+\frac{1}{p_2}$
and $\frac{1}{q_1}+\frac{1}{q_2}=1$. Then for any $\kappa_1\in(0,1)$ and $\kappa_2\in[-\kappa_1,0]$
there exists a constant $C=C(d,p,p_1,p_2,\kappa_1,\kappa_2)>0$ such that
\begin{equation}\label{ComEs}
\|[\cR_j,f]g\|_p\leq C2^{-j(\kappa_1+\kappa_2)}
\begin{cases}
\|f\|_{B_{p_1,\infty}^{\kappa_1}}\|g\|_{p_2}, & \kappa_2=0,\\
\|f\|_{B_{p_1,\infty}^{\kappa_1}}\|g\|_{B_{p_2,\infty}^{\kappa_2}}, & \kappa_1+\kappa_2>0,\\
\|f\|_{B_{p_1,q_1}^{\kappa_1}}\|g\|_{B_{p_2,q_2}^{\kappa_2}}, & \kappa_1+\kappa_2=0.
\end{cases}
\end{equation}
\end{lemma}

The following lemma is from \cite[Lemma 2.1]{GIP15}.
\begin{lemma}\label{0116Lem02}
\begin{itemize}
  \item [(i)] For $s>0$ and  $f,g\in \bC^{s}$, it holds that
  \begin{equation}\label{0306:02}
  \|fg\|_{\bC^s}\lesssim_C \|f\|_{\bC^s}\|g\|_{\bC^s}.
  \end{equation}
  \item [(ii)] For $s,\epsilon>0$ and  $f\in \bC^{-s}$,
  $g\in \bC^{s+\epsilon}$, it holds that
\begin{equation}\label{0306:03}
\|fg\|_{\bC^{-s}}
\lesssim \|f\|_{\bC^{-s}}\|g\|_{\bC^{s+\epsilon}}.
\end{equation}
\end{itemize}
\end{lemma}

\subsection{$\alpha$-stable processes}

Fix $\alpha\in(0,2)$.
Let $\{L_t^{(\alpha)},t\geq0\}$ be a rotationally invariant and symmetric $\alpha$-stable process
whose generator is defined by
\begin{equation*}
\sL\phi(x):=\int_{\R^d}\left(\phi(x+z)-\phi(x)-z^{(\alpha)}\cdot\nabla\phi(x)\right)\nu(\dif z),
\quad \phi\in\sS(\R^d),
\end{equation*}
where $z^{(\alpha)}:=z\b1_{\{\alpha\in[1,2)\}}\b1_{\{|z|\leq1\}}$ and the L\'evy measure
$$
\nu(\dif z)=\frac{c_{d,\alpha}}{|z|^{d+\alpha}}\dif z
$$
with some specific constant $c_{d,\alpha}>0$.
Note that for any $0\leq\kappa_1<\alpha<\kappa_2$,
\begin{align*}
\int_{|z|\leq1}|z|^{\kappa_2}\nu(\dif z)+\int_{|z|>1}|z|^{\kappa_1}\nu(\dif z)<\infty.
\end{align*}
Let $N(\dif s,\dif z)$ be the associated Poisson random measure defined by
\begin{equation*}
N((0,t]\times A):=\sum_{s\in[0,t]}\b1_A\left(L_s-L_{s-}\right)
\end{equation*}
for all $A\in\sB(\R^d\setminus\{0\})$ and $t>0$. Then it follows from L\'evy--It\^o's
decomposition (see e.g. \cite[Theorem 19.2]{Sato99}) that
\begin{equation*}
L_t^{(\alpha)}=\int_0^t\int_{|z|\leq1}z\widetilde{N}(\dif s,\dif z)
+\int_0^t\int_{|z|>1}zN(\dif s,\dif z),
\end{equation*}
where $\widetilde{N}(\dif s,\dif z):=N(\dif s,\dif z)-\nu(\dif z)\dif s$ is
the compensated Poisson random measure.

\section{Schauder's estimate for nonlocal PDEs}\label{SE}

In order to prove Schauder's estimate \eqref{SchM0124},
we firstly introduce the following frequency localized maximum principle
(see e.g. \cite[Lemma 3.2]{LZ22} and \cite[Lemma 4.7]{SX23})
and commutator estimate.
\begin{lemma}\label{LMP}
For any $j\geq0$, let
$$
\cA:=\left\{u\in\sS(\R^d):\supp\hat{u}\subset\{\xi:2^{j-1}\leq|\xi|\leq2^{j+1}\}\right\}.
$$
Under {\bf(H$_{\sigma}$)}, there exists a constant $C_1>0$ such that for all $j$ and $u\in\cA$,
\begin{equation*}
\inf_{x\in J(u)}\left\{\sgn(u(x))\cdot\left(-\sL_\sigma u(x)\right)\right\}
\geq C_12^{\alpha j}\|u\|_{\infty},
\end{equation*}
where $J(u):=\{x:|u(x)|=\|u\|_{\infty}\}$.
\end{lemma}
\begin{proof}
By \eqref{0414:03} and the change of variable,
we can rewrite the operator $\sL_\sigma$ as follows:
\begin{align*}
\sL_{\sigma}u(x)=\int_{\R^d}\left(u(x+z)-u(x)-z^{(\alpha)}\cdot\nabla u(x)\right)
\kappa(t,x,z)\dif z,
\end{align*}
where
$$
\kappa(t,x,z):=\left(|{\rm det}\sigma(t,x)||\sigma^{-1}(t,x)z|^{d+\alpha}\right)^{-1}.
$$
By {\bf(H$_{\sigma}$)}, there exists a constant $\kappa_0>0$ such that
for all $(t,x,z)\in[0,T]\times\R^{2d}$,
$$
\kappa(t,x,z)\geq \kappa_0/|z|^{d+\alpha}.
$$
Let $x_0\in J(u)$. Without loss of generality, we assume $u(x_0)=\|u\|_{\infty}>0$. Since $\nabla u(x_0)=0$, we have
\begin{align*}
\sL_\sigma u(x_0)
=\int_{\R^d}\left(u(x_0+z)-u(x_0)\right)
\kappa(t,x_0,z)\dif z
\leq \int_{\R^d}\left( u(x_0+z)-u(x_0)\right)\frac{\kappa_0}{|z|^{d+\alpha}}\dif z.
\end{align*}
By \cite[Lemma 3.2]{LZ22} (or \cite[Lemma 4.7]{SX23}), we conclude the proof.
\end{proof}

\begin{lemma}\label{Lem0329}
Under {\bf(H$_{\sigma}$)}, for any $\kappa\in[-1,0)$, there is a constant
$C=C(d,p,\alpha,\kappa,\Lambda_1)>0$
such that for any $\varphi\in\sS(\R^d)$, $j\geq-1$ and $\epsilon>0$,
\begin{align}\label{0329:01}
\|[\cR_j,\sL_\sigma]\varphi\|_\infty
\lesssim_C 2^{-j(1+\kappa)}\|\varphi\|_{B_{\infty,\infty}^{\alpha+\kappa+\epsilon}}.
\end{align}
\end{lemma}
\begin{proof}
 When $\alpha\in(0,1]$, it was proven in \cite[Lemma 4.5]{SX23}.
Below we  show it for $\alpha\in(1,2)$.
Note that by the definitions of $\cR_j$ and $\sL_\sigma$, we have for all $j\geq-1$,
\begin{align}\label{0329:02}
[\cR_j,\sL_\sigma]\varphi(x)
&
=\int_{\R^d}h_j(x-y)\int_{\R^d}\left(\varphi(y+\sigma(t,y)z)-\varphi(y)
-\sigma(t,y)z^{(\alpha)}\cdot\nabla \varphi(y)\right)\nu(\dif z)\dif y\nonumber\\\nonumber
&\quad
-\int_{\R^d}\Bigg(\int_{\R^d}h_j(y)\varphi(x+\sigma(t,x)z-y)\dif y
-\int_{\R^d}h_j(x-y)\varphi(y)\dif y\\
&\qquad
-\sigma(t,x)z^{(\alpha)}\cdot\int_{\R^d}h_j(y)\nabla \varphi(x-y)\dif y\Bigg)\nu(\dif z)\\\nonumber
&
=:\int_{\R^d}\Sigma(t,x,z)\nu(\dif z),
\end{align}
where for $(t,x,z)\in[0,T]\times\R^{2d}$ and $\sigma_t^x(\cdot):=\sigma(t,x-\cdot)$
and $\varphi^{x}(\cdot):=\varphi(x-\cdot)$,
$$
\Sigma(t,x,z)=\int_{\R^d}h_j(y)
\Bigg(\varphi^x(y-\sigma_t^x(y)z)
-\varphi^x(y-\sigma_t^x(0)z)
+\left(\sigma_t^x(y)-\sigma_t^x(0)\right)z^{(\alpha)}
\cdot\nabla \varphi^x(y)\Bigg)\dif y.
$$
Note that for any $k\in\N$ and $\theta\geq0$
\begin{equation}\label{0728:01}
\int_{\R^d}\left|\nabla^kh_j(y)\right|\left(|y|\wedge1\right)^{\theta}\dif y
\leq \int_{\R^d}\left|\nabla^kh_j(y)\right||y|^{\theta}\dif y
\lesssim 2^{(k-\theta)j}.
\end{equation}
Hence, in view of \eqref{0728:01}
and \cite[Lemma 2.2]{CHZ20}, we have
for any $\varpi_1\in[0,1)$, $|z|\leq 2^{-j}\wedge(1/2\Lambda_1^{-1})$
and $\epsilon>0$,
\begin{align}\label{0329:04}
\left|\Sigma(t,x,z)\right|
&
=\left|\int_{\R^d}h_j(y)\Bigg(\varphi^x(y-\sigma_t^x(y)z)
-\varphi^x(y-\sigma_t^x(0)z)
+\left(\sigma_t^x(y)-\sigma_t^x(0)\right)z^{(\alpha)}
\cdot\nabla \varphi^x(y)\Bigg)\dif y\right|\nonumber\\
&
\lesssim|z|^{1+\varpi_1}\|u\|_{B_{\infty,\infty}^{\varpi_1}}\int_{\R^d}\left(|h_j(y)|
+|\nabla h_j(y)|(|y|\wedge1)\right)\dif y\\
&\quad
+|z|^{1+\varpi_1}\|u\|_{B_{\infty,\infty}^{\varpi_1}}\left(\int_{\R^d}|\nabla h_j(y)|\left(|y|\wedge1\right)^{\varpi_1}\dif y\right)^{\varpi_1}
\left(\int_{\R^d}|h_j(y)|\left(|y|\wedge1\right)^{\varpi_1}\dif y\right)^{1-\varpi_1}\nonumber\\
&
\lesssim |z|^{1+\varpi_1}\|\varphi\|_{B_{\infty,\infty}^{\varpi_1}}
= |z|^{1+\varpi_1-\alpha-\epsilon}  |z|^{\alpha+\epsilon}
\|\varphi\|_{B_{\infty,\infty}^{\varpi_1}}.\nonumber
\end{align}
Note that we also have
\begin{equation*}
\Sigma(t,x,z)
=\int_{\R^d}h_j(y)\Bigg(\varphi^{x}(y-\sigma_t^{x}(y)z)
-\varphi^{x}(y-\sigma_t^{x}(0)z)\Bigg)\dif y
-z^{(\alpha)}[\cR_j,\sigma]\nabla \varphi(x),
\end{equation*}
which implies, by {\bf(H$_{\sigma}$)}, \eqref{0301:01} and \eqref{ComEs},
that for any $\iota\in(0,1)$, $\epsilon>0$ and $|z|>2^{-j}\wedge(1/2\Lambda_1^{-1})$,
\begin{align}\label{0329:05}
\left|\Sigma(t,x,z)\right|
&
\leq\int_{\R^d}|h_j(y)|\|\varphi\|_{B_{\infty,\infty}^\iota}
\left(\|\nabla\sigma\|_\infty^\iota|yz|^\iota\wedge1\right)\dif y
+2^{-j\iota}|z|\|\sigma\|_{\bC^1}\|\varphi\|_{B_{\infty,1}^\iota}\\\nonumber
&
\lesssim 2^{-\iota j}\|\varphi\|_{B_{\infty,\infty}^{\iota+\epsilon}}\left(|z|^\iota+|z|\right).
\end{align}
Combining \eqref{0329:02}, \eqref{0329:04} and \eqref{0329:05}, and choosing
$\varpi_1=\kappa+\alpha+\epsilon$ and $\iota=\alpha+\kappa$, we have
\begin{align}\label{0429:01}
\|[\cR_j,\sL_\sigma]\varphi\|_\infty
&
\lesssim 2^{-j(1+\kappa)}\|\varphi\|_{B_{\infty,\infty}^{\alpha+\kappa+\epsilon}}
\int_{\R^d}|z|^{\alpha+\eps}\wedge(|z|^{\alpha+\kappa}+|z|)\nu(\dif z)\nonumber\\
&
\lesssim 2^{-j(1+\kappa)}\|\varphi\|_{B_{\infty,\infty}^{\alpha+\kappa+\epsilon}}
\end{align}
for any $\kappa\in[-1,1-\alpha)$ and $\epsilon>0$.

On the other hand, by \eqref{0301:01} we have for any $\varpi_2\in(\alpha-1,1)$ and $|z|<1$,
\begin{align}\label{0611:01}
\left|\Sigma(t,x,z)\right|
&
=\left|\int_{\R^d}h_j(y)\Bigg(\varphi^x(y-\sigma_t^x(y)z)-\varphi^x(y-\sigma_t^x(0)z)
+\left(\sigma_t^x(y)-\sigma_t^x(0)\right)z\cdot\nabla \varphi^x(y)\Bigg)\dif y\right|\nonumber\\
&
\lesssim\int_{\R^d}|h_j(y)|\left|\int_0^1\left[\nabla \varphi^x\left(y-\theta\sigma_t^x(y)z
-(1-\theta)\sigma_t^x(0)z\right)-\nabla \varphi^x(y)\right]
\left(\sigma_t^x(y)-\sigma_t^x(0)\right)z\dif\theta\right|\nonumber\\
&
\lesssim \int_{\R^d}|h_j(y)|\|\nabla \varphi\|_{B_{\infty,\infty}^{\varpi_2}}|z|^{1+\varpi_2}|y|\dif y\\
&
\lesssim 2^{-j}|z|^{1+\varpi_2}\|\varphi\|_{B_{\infty,\infty}^{1+\varpi_2}}.\nonumber
\end{align}
For any $\epsilon>0$,
in view of \eqref{0329:05} with $\iota=1-$ and  \eqref{0611:01}
with $\varpi_2=\alpha-1+\epsilon$, we obtain that
\begin{equation*}
\left\|[\cR_j,\sL_\sigma]\varphi\right\|_\infty
\lesssim 2^{-j\iota}\|\varphi\|_{B_{\infty,\infty}^{\alpha+\eps}}
\int_{\R^d}|z|\wedge|z|^{\alpha+\epsilon}\nu(\dif z)
\lesssim 2^{-j\iota}\|\varphi\|_{B_{\infty,\infty}^{\alpha+\eps}},
\end{equation*}
which by interpolation and \eqref{0429:01} completes the proof.
\end{proof}

By Lemmas \ref{Lem28} and \ref{Lem0329}, we have the following lemma.
For readers' convenience, we also give the proof.
\begin{lemma}\label{lem0728}
Assume that {\bf(H$_{\sigma}$)} holds. Then for any $\iota\in(-1,0]$,
there exists a constant $C=C(d,p,\alpha,\kappa,\Lambda_1)>0$ such that
for all $\varphi\in \sS(\R^d)$,
\begin{equation*}
\|\sL_{\sigma}\varphi\|_{B_{\infty,\infty}^\iota}
\lesssim_C\|\varphi\|_{B_{\infty,\infty}^{\alpha+\iota}}.
\end{equation*}
\end{lemma}
\begin{proof}
Let
$$
\Sigma_1(t,x,z):=\cR_j\varphi(x+\sigma(t,x)z)-\cR_j\varphi(x),\quad
\Sigma_2(t,x,z):=\Sigma_1(t,x,z)-\sigma(t,x)z\cdot\nabla\cR_j\varphi(x).
$$
If $\alpha\in(1,2)$, then by Taylor’s expansion and Lemma \ref{Lem28} we have
\begin{align}\label{0728:02}
\left|\sL_\sigma\cR_j\varphi(x)\right|
&
=\left|\int_{|z|\leq1}\Sigma_2(t,x,z)\nu(\dif z)
+\int_{|z|>1}\Sigma_1(t,x,z)\nu(\dif z)\right|\no\\
&
\lesssim \|\sigma\|_{L_T^\infty} \int_{|z|\leq1}\left(\|\nabla^2\cR_j\varphi\|_\infty|z|^2\right)
\wedge \|\cR_j\varphi\|_\infty\nu(\dif z)
+\int_{|z|>1}\|\cR_j\varphi\|_\infty\nu(\dif z)\\
&
\lesssim \left(\|\sigma\|_{L_T^\infty}+1\right)\|\cR_j\varphi\|_\infty
\left(2^{j\alpha}\int_{|z|\leq1}|z|^{\alpha}\nu(\dif z)
+\int_{\R^d}1\nu(\dif z)\right)\no\\
&
\lesssim \left(\|\sigma\|_{L_T^\infty}+1\right)2^{j\alpha}\|\cR_j\varphi\|_\infty.\no
\end{align}
Similarly, we also obtain that \eqref{0728:02} holds when $\alpha\in(0,1]$.
Combining \eqref{0728:02} and Lemma \ref{Lem0329}, we have for any
$\iota\in(-1,0]$,
\begin{equation*}
\|\sL_{\sigma}\varphi\|_{B_{\infty,\infty}^\iota}
\lesssim\|\varphi\|_{B_{\infty,\infty}^{\alpha+\iota}}.
\end{equation*}
\end{proof}

Now we are in the position to prove Theorem \ref{SchE24}.
\begin{proof}[Proof of Theorem \ref{SchE24}]
By standard mollifying approximation, it suffices to show
the a priori estimate \eqref{SchM0124}.
For simplicity, we omit the subscript $\lambda$ of $u_\lambda$ below.
Since for each $t\in[0,T]$ and $j$,
$\cR_j u(t,\cdot)\in\sS(\R^d)$, there exists an $x_{t,j}\in\R^d$ such that
$$|\cR_j u(t,x_{t,j})|=\|\cR_ju(t,\cdot)\|_\infty.$$
Without loss of generality we assume
$\cR_j u(t,x_{t,j})=\|\cR_ju(t,\cdot)\|_\infty$.
By \cite[Lemma 3.2]{WZ11}, one has
$$
\p_t\|\cR_ju(t)\|_\infty=(\p_t\cR_ju)(t,x_{t,j})
$$
Now, using the operator $\cR_j$ to act on both sides of \eqref{Pareq24}, and by Lemma \ref{LMP}, we get
\begin{align}
\p_t\|\cR_ju(t)\|_\infty
&
=\sL_\sigma\cR_ju(t,x_{t,j})-\lambda\cR_ju(t,x_{t,j})+[\cR_j,\sL_\sigma]u(t,x_{t,j})\nonumber\\
&\quad
+\cR_j(g\cdot\nabla u)(t,x_{t,j})+\cR_jf(t,x_{t,j})\no\\
&\leq -(C_1 2^{\alpha j}+\lambda)\|\cR_ju(t)\|_\infty+\|[\cR_j,\sL_\sigma]u(t)\|_\infty\no\\
&\qquad+\|\cR_jf(t)\|_\infty+\cR_j(g\cdot\nabla u)(t,x_{t,j}),
\label{0310:01}
\end{align}
which by  Gronwall's inequality implies that
\begin{align}\label{0310:02}
\|\cR_ju(t)\|_\infty
&
\lesssim\int_0^t\e^{-(C_12^{j\alpha}+\lambda)(t-s)}\left(\|[\cR_j,\sL_\sigma]u(s)\|_\infty
+\|\cR_jf(s)\|_{\infty}\right)\dif s\nonumber\\
&\quad
+\int_0^t\e^{-(C_12^{j\alpha}+\lambda)(t-s)}|\cR_j(g\cdot\nabla u)(s,x_{s,j})|\dif s=:\sI_1+\sI_2.
\end{align}
Let $\vartheta\in(-\infty,1)$ and $\epsilon>0$. By \eqref{0329:01} with
$\kappa=\vartheta\vee 0-1$, one sees that
\begin{align*}
\|[\cR_j,\sL_\sigma]u\|_\infty
\lesssim 2^{-(\vartheta\vee 0) j}\|u(s)\|_{B_{\infty,\infty}^{\alpha-1+\vartheta\vee 0+\epsilon}}
\lesssim 2^{-\vartheta j}\|u(s)\|_{B_{\infty,\infty}^{\alpha-1+\vartheta\vee 0+\epsilon}},
\end{align*}
and by definition of $\bC^\vartheta$,
$$
\|\cR_jf(s)\|_{\infty}\leq 2^{-j\vartheta}\|f(s)\|_{\bC^\vartheta}.
$$
Therefore,
\begin{align}\label{0108:03}
\|\cR_ju(t)\|_\infty
\lesssim2^{-j\vartheta}\int_0^t\e^{-(C_12^{j\alpha}+\lambda)(t-s)}
\Big(\|u(s)\|_{B_{\infty,\infty}^{\alpha-1+\vartheta\vee 0+\epsilon}}
+\|f(s)\|_{\bC^\vartheta}\Big)\dif s+\sI_2.
\end{align}
Now let us treat $\sI_2$. When $\beta>0$,  using the fact
$$
\cR_j\nabla u(t,x_{t,j})=\nabla \cR_ju(t,x_{t,j})=0,
$$
for any $\vartheta\in(-\infty,\beta)$ and  $\epsilon>0$,
by \eqref{ComEs} with $\kappa_1=\beta$ and $\kappa_2=(\vartheta\vee0)-\beta$, we have
\begin{align*}
|\cR_j(g\cdot\nabla u)(s,x_{s,j})|=|[\cR_j,g]\nabla u(s,x_{s,j})|\lesssim 2^{-(\vartheta\vee0)j}\|g\|_{L^\infty_T(\bC^\beta)}
\|u(s)\|_{B_{\infty,\infty}^{1-\beta+(\vartheta\vee0)+\epsilon}}.
\end{align*}
When $\beta\le0$, for any $\vartheta\le\beta$, by \eqref{0306:03} we have
\begin{align*}
|\cR_j(g\cdot\nabla u)(s,x_{s,j})|
\le 2^{-\vartheta j}\|g\cdot\nabla u(s)\|_{\bC^{\vartheta}}
\lesssim 2^{-\vartheta j}\|g\cdot\nabla u(s)\|_{\bC^{\beta}}
\lesssim 2^{-\vartheta j}\|g\|_{L^\infty_T(\bC^\beta)}\|u(s)\|_{B_{\infty,\infty}^{1-\beta+\epsilon}}.
\end{align*}
Thus, for any $\vartheta\le \beta<1$, letting $\delta:=(\alpha-1)\vee(1-\beta)+(\vartheta\vee0)$,
by \eqref{0108:03} we have
\begin{align}\label{0430:00}
\|\cR_ju(t)\|_\infty
\lesssim2^{-j\vartheta}\int_0^t\e^{-(C_12^{j\alpha}+\lambda)(t-s)}
\Big(\|u(s)\|_{B_{\infty,\infty}^{\delta+\epsilon}}+\|f(s)\|_{\bC^\vartheta}\Big)\dif s.
\end{align}
Since for any $\eta\in[0,\alpha]$,
 $$
  \int_0^t\e^{-(C_12^{j\alpha}+\lambda)(t-s)}\dif s\lesssim \left[2^{-j\alpha}\wedge(1+\lambda)^{-1}\right]\le (1+\lambda)^{-\frac{\alpha-\eta}{\alpha}}2^{-j(\vartheta+\eta)},
  $$
taking supremum for $j$ in \eqref{0430:00}, we get
\begin{align}\label{0430:01}
\|u\|_{L^\infty_T(B_{\infty,\infty}^{\vartheta+\eta})}\lesssim (1+\lambda)^{-\frac{\alpha-\eta}{\alpha}}\left(\|u\|_{L^\infty_T(B_{\infty,\infty}^{\delta+\epsilon})}
+\|f\|_{L^\infty_T(\bC^\vartheta)}\right).
\end{align}
On the other hand, noting that for any $\epsilon\in(0,\alpha)$,
$$
\e^{-C_12^{j\alpha}(t-s)}\lesssim 2^{-j(\alpha-\epsilon)}(t-s)^{-\frac{\alpha-\epsilon}{\alpha}},
$$
by \eqref{0430:00}, we have for any $\epsilon\in(0,\alpha)$,
\begin{align*}
\|\cR_ju(t)\|_\infty
&
\lesssim 2^{-j\vartheta}\int_0^t\e^{-C_12^{j\alpha}(t-s)}
\|u(s)\|_{B_{\infty,\infty}^{\delta+\epsilon}}\dif s
+2^{-(\alpha+\vartheta)j}\|f\|_{L^\infty_T(\bC^\vartheta)}\\
&
\lesssim 2^{-j(\vartheta+\alpha-\epsilon)}\left[\int_0^t(t-s)^{-\frac{\alpha-\epsilon}{\alpha}}
\|u(s)\|_{B_{\infty,\infty}^{\delta+\epsilon}}\dif s
+\|f\|_{L^\infty_T(\bC^\vartheta)}\right],
\end{align*}
which implies that
\begin{align*}
\|u(t)\|_{B_{\infty,\infty}^{\alpha+\vartheta-\epsilon}}
\lesssim \int_0^t(t-s)^{-\frac{\alpha-\epsilon}{\alpha}}
\|u(s)\|_{B_{\infty,\infty}^{\delta+\epsilon}}\dif s
+\|f\|_{L^\infty_T\bC^\vartheta}.
\end{align*}
Let $\vartheta>(1-\alpha-\beta)\vee(-1)$. Since $\alpha>1-\beta$, we can choose $\epsilon>0$ small enough so that
\begin{align*}
\begin{cases}
\alpha+\vartheta-\epsilon>(1-\beta)\vee(\alpha-1)+\epsilon=\delta+\epsilon,\quad \vartheta\le 0;\\
\alpha+\vartheta-\epsilon>(\alpha-1)\vee(1-\beta)+\vartheta+\epsilon=\delta+\epsilon,\quad \vartheta>0.
\end{cases}
\end{align*}
Thus by Gronwall's inequality of Volterra's type, we get
\begin{align*}
\|u(t)\|_{B_{\infty,\infty}^{\delta+\epsilon}}
\lesssim \|u(t)\|_{B_{\infty,\infty}^{\alpha+\vartheta-\epsilon}}
\lesssim  \|f\|_{L^\infty_T(\bC^\vartheta)}.
\end{align*}
Substituting this into \eqref{0430:01}, we obtain \eqref{SchM0124}.

If $\sigma$ is a constant matrix, then there is no commutator term in \eqref{0310:02}.
Instead of \eqref{0430:01}, we have
$$
\|u\|_{L^\infty_T(B_{\infty,\infty}^{\vartheta+\eta})}
\lesssim (1+\lambda)^{-\frac{\alpha-\eta}{\alpha}}
\left(\|u\|_{L^\infty_T(B_{\infty,\infty}^{1-\beta+(\vartheta\vee0)+\eps})}
+\|f\|_{L^\infty_T(\bC^\vartheta)}\right).
$$
and for any $\vartheta>1-\alpha-\beta$, one can choose $\eps>0$ so that $\alpha+\vartheta-\eps>1-\beta+(\vartheta\vee0)+\eps$ and
\begin{align*}
\|u(t)\|_{B_{\infty,\infty}^{1-\beta+(\vartheta\vee0)+\eps}}
\lesssim \int_0^t(t-s)^{-\frac{\alpha-\epsilon}{\alpha}}
\|u(s)\|_{B_{\infty,\infty}^{1-\beta+(\vartheta\vee0)+\eps}}\dif s
+\|f\|_{L^\infty_T(\bC^\vartheta)}.
\end{align*}
By Gronwall's inequality, we conclude the proof.
\end{proof}

By Theorem \ref{SchE24}, we have the following easy corollary.
\begin{lemma}\label{BWPDE}
Let $g\in L_T^\infty(\bC^\beta)$ and $f\in L_T^\infty(\bC^\vartheta)$.
Fix $t\in[0,T]$. Consider the backward PDE
\begin{equation}\label{0310:03}
\partial_ru^t(r,x)+\sL_{\sigma(r,x)}u^t(r,x)+g\cdot \nabla u^t(r,x)+f(r,x)=0,
\quad u^t(t,x)=0.
\end{equation}
Assume that $(\alpha,\beta,\vartheta)$ is as in Theorem \ref{SchE24}.
Then there exists a unique solution $u^t$ to \eqref{0310:03} such that
for any
$\eta\in[0,\alpha]$,
\begin{equation}\label{0310:04}
\|u^t\|_{L_{[0,t]}^\infty(B_{\infty,\infty}^{\vartheta+\eta})}
\lesssim 
\|f\|_{L_T^\infty(\bC^\vartheta)}.
\end{equation}
\end{lemma}
\begin{proof}
Define $\widetilde{u}^t(r,x):=u^t(t-r,x)$,
$\widetilde{\sigma}^t(r,x):=\sigma(t-r,x)$, $\widetilde{g}^t(r,x):=g(t-r,x)$
and $\widetilde{f}^t(r,x):=f(t-r,x)$
for all $(r,x)\in[0,t]\times\R^d$.
It can be verified that $\widetilde{u}^t$ satisfies the following forward PDE on $0\leq r\leq t$:
\begin{equation*}
 \partial_r\widetilde{u}^t(r,x)=\sL_{\widetilde{\sigma}^t}\widetilde{u}^t(r,x)
 +\widetilde{g}^t\cdot\nabla\widetilde{u}^t(r,x)+\widetilde{f}^t(r,x),\quad \widetilde{u}^t(0,x)=0.
\end{equation*}
Thanks to Theorem \ref{SchE24}, we obtain that
\begin{align*}
\|\widetilde{u}^t\|_{L_{[0,t]}^\infty(B_{\infty,\infty}^{\vartheta+\eta})}
&
\lesssim \|f(t-\cdot,\cdot)\|_{L_{[0,t]}^\infty(\bC^\vartheta)}
\lesssim \|f\|_{L_T^\infty(\bC^\vartheta)},
\end{align*}
which completes the proof. 
\end{proof}

Now we prove the following crucial Krylov-type estimate.
\begin{lemma}\label{KryEq}
Assume that {\bf(H$_{\sigma}$)} holds.
Let $(\alpha,\beta)\in\cA_1\cup\cA_2=:\cA$, $T>0$, $g\in L^\infty_T(\bC^{\beta})$,
and $X_t,t\in[0,T]$ be the solution of the following SDE
\begin{align*}
\dif X_t=g(t,X_t)\dif t+\sigma(t,X_{t-})\dif L_t^{(\alpha)}.
\end{align*}
Then for any $\delta\in(0,\alpha/2\wedge(\alpha+\beta-1))\cap[-\beta,\infty)$,
$f\in L^\infty_T(\bC^{\kappa}),\kappa>0$ and $p\geq2$,
there is a constant
$
C=C(\alpha,T,\beta,\|g\|_{L^\infty_T(\bC^{\beta})},\delta,p)>0
$
such that
\begin{align}\label{B03}
\mE\left[\sup_{t\in[0,T]}\left|\int_0^t f(s,X_s)\dif s\right|^p\right]\le
C\|f\|_{L^\infty_T(\bC^{-\delta})}^p.
\end{align}
\end{lemma}

\begin{proof}
Consider the following backward PDE for $0\leq s\leq T$:
\begin{align*}
\p_s u^T+\sL_\sigma u^T+g\cdot\nabla u^T+f=0,\quad u^T(T,x)=0.
\end{align*}
By Lemma \ref{BWPDE} and applying It\^o's formula to $s\mapsto u^T(s,X_{s})$, we have
\begin{align*}
\int_0^t f(s,X_s)\dif s=u^T(0,X_0)-u^T(t,X_t)
-\int_0^t\int_{\mR^d}
\left(u^T(s,X_{s-}+\sigma(s,X_{s-})z)-u^T(s,X_{s-})\right)\tilde{N}(\dif s,\dif z).
\end{align*}
Then by \eqref{0310:04} and the Burkholder-Davis-Gundy inequality
(\cite[Theorem 4.4.23]{APP09}),
we have for any $0<\epsilon<2-\alpha$,
\begin{align*}
&
\mE\left[\sup_{t\in[0,T]}\left|\int_0^t f(s,X_s)\dif s\right|^p\right]\\
&
\lesssim \|u^T\|_{L^\infty_T}^p
+\mE\left(\int_0^T\int_{\mR^d}|u^T(s,X_s+\sigma(s,X_s)z)-u^T(s,X_s)|^2\nu(\dif z)\dif s\right)^{p/2}\\
&\quad
+\mE\int_0^T\int_{\R^d}|u^T(s,X_s+\sigma(s,X_s)z)-u^T(s,X_s)|^p\nu(\dif z)\dif s\\
&
\lesssim \|u^T\|_{L^\infty_T}^p
+\|u^T\|_{L^\infty_T(\bC^{(\alpha+\epsilon)/2})}^p\left(\int_{\mR^d}
(1\wedge|z|^{\alpha+\epsilon})\nu(\dif z)\right)^{p/2}\lesssim \|f\|_{L^\infty_T(\bC^{-\delta})}^p,
\end{align*}
which completes the proof.
\end{proof}

\begin{lemma}\label{ElLem24}
Let $\sigma(t,x)=\sigma(x)$ for all $(t,x)\in\R_+\times\R^d$.
Consider the following elliptic
equation:
\begin{equation}\label{Elleq24}
\sL_\sigma v_\lambda-\lambda v_\lambda+g\cdot\nabla v_\lambda=f,
\end{equation}
where $g\in\bC^{\beta}$ and $f\in \bC^\vartheta$.
If $\beta$ and $\vartheta$ are as in Theorem \ref{SchE24},
then there exists a constant $\lambda_0>0$ such that for any $\lambda\geq\lambda_0$
there is a unique solution
to \eqref{Elleq24} such that for any $\eta\in[0,\alpha]$,
\begin{equation*}
\|v_\lambda\|_{B_{\infty,\infty}^{\vartheta+\eta}}
\leq C(\lambda)\|f\|_{\bC^\vartheta},
\end{equation*}
where $C(\lambda)$, depending on $\Theta,\|g\|_{\bC^\beta},\eta$
and $\lambda$,
satisfies $\lim_{\lambda\rightarrow\infty}C(\lambda)=0$.
\end{lemma}
\begin{proof}
Let $v_\lambda$ be a solution to \eqref{Elleq24} and $\phi(t),t\geq0$ be a nonnegative
and nonzero smooth function with $\phi(0)=0$.
Then $u_\lambda(t,x):=\phi(t)v_\lambda(x)$ satisfies
\begin{equation*}
\partial_tu_\lambda(t,x)=\sL_{\sigma} u_\lambda(t,x)-\lambda u_\lambda(t,x)
+g\cdot\nabla u_\lambda(t,x)-\phi(t)f(x)+\phi'(t)v_\lambda(x),\quad u_\lambda(0,x)=0.
\end{equation*}
In view of Theorem \ref{SchE24}, we have for any $\eta\in[0,\alpha]$,
\begin{align*}
\|u_\lambda\|_{L_T^\infty(B_{\infty,\infty}^{\vartheta+\eta})}
\leq C(1+\lambda)^{-\frac{\alpha-\eta}{\alpha}}
\|\phi f-\phi'v_\lambda\|_{L_T^\infty(\bC^{\vartheta})},
\end{align*}
where $C=C(\Theta,\|g\|_{\bC^\beta},\eta)$.
Taking $\lambda_0\leq\lambda$,
then we have
\begin{align*}
\|v_\lambda\|_{\bC^{\vartheta+\eta}}
\leq
&
C(1+\lambda_0)^{-\frac{\alpha-\eta}{2\alpha}}(1+\lambda)^{-\frac{\alpha-\eta}{2\alpha}}
\|\phi\|_\infty^{-1}\left(\|\phi\|_\infty\|f\|_{\bC^\vartheta}
+\|\phi'\|_\infty\|v_\lambda\|_{\bC^{\vartheta}}\right).
\end{align*}
Letting $\lambda_0$ large enough such that
$$
C(1+\lambda_0)^{-\frac{\alpha-\eta}{2\alpha}}\|\phi\|_\infty^{-1}\|\phi'\|_\infty<\frac12,
$$
we complete the proof.
\end{proof}

\section{Averaging principle}\label{AP}
\subsection{Proof of Theorem \ref{Th1}}\label{SCR}

Before proving Theorem \ref{Th1}, let us first show some results concerning  mollifying
approximation.
Let $\rho\in\sS(\R^d)$ be a probability density function.
For any $n\in\N$, we set
\begin{equation*}
\rho_n(x):=n^d\rho(nx),\quad \forall x\in\R^d.
\end{equation*}
Then define the mollifying approximation of
$f\in\sS'(\R^d)$ as follows:
\begin{equation*}
f_n(x):=f*\rho_n(x),\quad \forall x\in\R^d.
\end{equation*}

\begin{lemma}\label{UConB}
Assume that $f\in B_{\infty,\infty}^{\kappa}$, where $\kappa\in\R$. Then for any $\delta\geq0$,
\begin{equation}\label{UConBeq}
\|f_n\|_{B_{\infty,\infty}^{\kappa+\delta}}
\lesssim \|f\|_{B_{\infty,\infty}^{\kappa}}n^\delta.
\end{equation}
\end{lemma}
\begin{proof}
Fix $j\geq0$ and $m\geq0$. Note that by changes of variables, we have
\begin{align*}
\|\cR_j\rho_n\|_{L^1}
&
=\int_{\R^d}\left|\int_{\R^d}h_j(y)\rho_n(x-y)\dif y\right|\dif x\\\nonumber
&
=2^{-jd}\int_{\R^d}\left|\int_{\R^d}\rho_n(2^{-j}(x-y))h(y)\dif y\right|\dif x\\\nonumber
&
\lesssim2^{-jd}\int_{\R^d}\left|\int_{\R^d}\nabla^m\rho_n(2^{-j}(x-y))\nabla^{-m}h(y)\dif y\right|\dif x\\\nonumber
&
\lesssim 2^{-jd}\int_{\R^d}\int_{\R^d}\left|\nabla^m\rho_n(2^{-j}(x-y))\right|
\left|\nabla^{-m}h(y)\right|\dif x\dif y\\\nonumber
&
\lesssim 2^{-jm}\int_{\R^d}\left|\nabla^m\rho_n(x)\right|\dif x\lesssim 2^{-jm}n^{m},
\end{align*}
which together with the property of $\cR_j$ implies that
\begin{align}\label{UConBeq02}
\|\cR_j\rho_n\|_{L^1}
\lesssim \|\rho_n\|_{L^1}\wedge2^{-jm}n^{m}
\lesssim1\wedge \left(2^{-jm}n^{m}\right),
\quad j,m\geq0.
\end{align}
Without loss of generality, we assume that $\delta>0$.
Then it follows from \eqref{UConBeq02} and a change of variable
that for any $\delta>0$ and $m>1$
\begin{align*}
\|\rho_n\|_{B_{1,1}^\delta}
&
=\sum_{j=-1}^\infty2^{j\delta}\|\cR_j\rho_n\|_{L^1}
\lesssim 1+\sum_{j=1}^\infty2^{j\delta}\|\cR_j\rho_n\|_{L^1}\\
&
\lesssim 1+\sum_{j=1}^\infty2^{j\delta}\left[1\wedge2^{-jm}n^{m}\right]\\
&
\lesssim 1+\sum_{j=1}^\infty2^{j(\delta-1)}[1\wedge 2^{-jm}n^{m}]\int_{2^{j-1}}^{2^j}1\dif x\\
&
\lesssim 1+\int_0^\infty x^{\delta-1}[1\wedge x^{-m}n^m]\dif x\\
&
\lesssim 1+n^\delta\int_0^\infty y^{\delta-1}[1\wedge y^{-m}]\dif y\lesssim n^\delta,
\end{align*}
which implies that
\begin{align*}
\|f_n\|_{B_{\infty,\infty}^{\kappa+\delta}}
&
=\sup_{j\geq-1}2^{j(\kappa+\delta)}\|\cR_jf_n\|_\infty\\
&
\leq\sup_{j\geq-1}\left(2^{j(\kappa+\delta)}\sum_{i=j-1}^{j+1}\|\cR_jf*\cR_i\rho_n\|_\infty\right)\\
&
\leq\sup_{j\geq-1}\left(2^{j\kappa}\|\cR_jf\|_\infty2^\delta\sum_{i=j-1}^{j+1}2^{i\delta}
\|\cR_i\rho_n\|_{L^1}\right)\\
&
\leq 2^\delta\|f\|_{B_{\infty,\infty}^\kappa}\|\rho_n\|_{B_{1,1}^\delta}
\lesssim\|f\|_{B_{\infty,\infty}^\kappa}n^\delta.
\end{align*}
The proof is complete.
\end{proof}

\begin{lemma}\label{UConA}
Suppose that $f\in \bC^{\kappa}$, $\kappa\in\R$.
Then for any $\delta\in[0,1]$ we have
\begin{equation}\label{UConAeq}
\|f_n-f\|_{B_{\infty,\infty}^{\kappa}}
\lesssim n^{-\delta}\|f\|_{B_{\infty,\infty}^{\kappa+\delta}}.
\end{equation}
\end{lemma}
\begin{proof}
If $\delta=0$ then \eqref{UConAeq} follows from Lemma \ref{UConB}.
Now we consider the case where $\delta\in(0,1]$.
Given $y\in\R^d$, we have for all $\delta\in(0,1)$,
\begin{align}\label{UconAP01}
\|f(\cdot-y)-f(\cdot)\|_{B_{\infty,\infty}^{\kappa}}
&
=\sup_{j\geq-1}2^{j\kappa}\|\cR_jf(\cdot-y)-\cR_jf(\cdot)\|_\infty\nonumber\\
&
\leq \sup_{j\geq-1}2^{j\kappa}|y|^{\delta}\|\cR_jf\|_{\bC^{\delta}}\\
&
\leq \sup_{j\geq-1}2^{j\kappa}|y|^{\delta}\|\cR_jf\|_{B_{\infty,\infty}^{\delta}}\no\\
&
=\sup_{j\geq-1}\left[2^{j\kappa}|y|^{\delta}\left(\sup_{j-1\leq i\leq j+1}
2^{i\delta}\|\cR_i\cR_jf\|_\infty\right)\right]\nonumber\\
&
\leq 4|y|^{\delta}\|f\|_{B_{\infty,\infty}^{\kappa+\delta}},\nonumber
\end{align}
and by \eqref{Berneq01} we have for $\delta=1$,
\begin{align}\label{0802:01}
\|f(\cdot-y)-f(\cdot)\|_{B_{\infty,\infty}^{\kappa}}
&
=\sup_{j\geq-1}2^{j\kappa}\|\cR_jf(\cdot-y)-\cR_jf(\cdot)\|_\infty\nonumber\\
&
\leq \sup_{j\geq-1}2^{j\kappa}|y|\|\nabla\cR_jf\|_{\infty}\\
&
\leq \sup_{j\geq-1}2^{j(\kappa+1)}|y|\|\cR_jf\|_{\infty}
=|y|\|f\|_{B_{\infty,\infty}^{\kappa+1}}.\nonumber
\end{align}
Then by \eqref{UconAP01} and \eqref{0802:01} one sees that
for all $\delta\in(0,1]$
\begin{align*}
\|f_n-f\|_{B_{\infty,\infty}^{\kappa}}
&
=\left\|\int_{\R^d}\left(f(\cdot-y)-f(\cdot)\right)\rho_n(y)\dif y\right\|_{B_{\infty,\infty}^{\kappa}}\\\nonumber
&
\leq \int_{\R^d}\|f(\cdot-y)-f(\cdot)\|_{B_{\infty,\infty}^{\kappa}}|\rho_n(y)|\dif y\\\nonumber
&
\lesssim \int_{\R^d}|y|^{\delta}|\rho_n(y)|\dif y\|f\|_{B_{\infty,\infty}^{\kappa+\delta}}\\
&
\lesssim n^{-\delta}\|f\|_{B_{\infty,\infty}^{\kappa+\delta}}.
\end{align*}
\end{proof}

Let $\{Z_t,t\geq0\}$ be a nonnegative c\`adl\`ag process and $\{A_t,t\geq0\}$
be an increasing predictable process.
We say that {\it $A$ dominates $Z$} if for any stopping
time $\tau$,
\begin{equation*}
\mE Z_\tau\leq\mE A_\tau.
\end{equation*}
Recall the following Lenglart's inequality; see \cite[Corollary II]{Len77}.
\begin{lemma}\label{Leneq}
If $Z$ is dominated by $A$, then for any $p\in(0,1)$
and stopping time $\tau$,
\begin{equation*}
\mE\left[\left(\sup_{s\in[0,\tau]}|Z_s|\right)^p\right]
\leq\frac{2-p}{1-p}\mE\left(A_\tau^p\right).
\end{equation*}
\end{lemma}

For simplicity, we write for a function $f:\R_+\times\R^d\rightarrow\R^m$
$$
f_\eps(t,x):=f\left(\tfrac{t}{\eps},x\right),\quad \forall (t,x)\in\R_+\times\R^d.
$$
Now we give the following crucial lemma, which describes the fluctuation
between $b$ and $\bar{b}$ when $\phi=b-\bar{b}$.
\begin{lemma}\label{0116Lem}
Assume that $\phi\in L_T^\infty(\bC^2)$
and $X_t^\eps$ is the solution to \eqref{Main01} with $(\alpha,\beta)\in\cA_1$.
Then for any $p\geq1$ and $\kappa>\left(1-\tfrac{\alpha}{2}\right)\vee
\tfrac{\alpha}{2}$, there exits a constant
$C=C(\Theta,p,\|b\|_{L_T^\infty(\bC^\beta)})$
such that for any $0<\eps\leq1$,
\begin{equation*}
\mE\left(\sup_{t\in[0,T]}\left|\int_0^t\phi\left(\frac{s}{\eps},X_s^\eps\right)\dif s\right|^p\right)
\lesssim_C \eps^p\sup_{t\in[0,T]}\left\|\int_0^{\frac{t}{\eps}}\phi(s,\cdot)\dif s\right\|_{\bC^\kappa}^p.
\end{equation*}
\end{lemma}
\begin{proof}
Define
\begin{align*}
\psi(t,x):=\int_0^t \phi(s,x)\dif s,\quad \forall (t,x)\in[0,T]\times\R^d.
\end{align*}
It follows from It\^o's formula that
\begin{align}\label{0521:01}
\psi_\eps(t,X^\eps_t)=
\int_0^t \left[\left(
\p_s\psi_\eps+b_\eps\cdot\nabla\psi_\eps
+\sL_{\sigma_\eps}\psi_\eps\right)(s,\cdot)(X^\eps_s)\right]\dif s+M_t^\eps,
\end{align}
where
$$
M_t^\eps:=\int_0^t\int_{\mR^d}\left(\psi_\eps(s,X^\eps_{s-}+\sigma_\eps(s,X_{s-}^\eps)z)
-\psi_{\eps}(s,X^\eps_{s-})\right)\tilde{N}(\dif s,\dif z).
$$
Note that
$$\p_s\psi_\eps(s,x)=\frac{1}{\eps}\phi_\eps(s,x),~\forall(s,x)\in[0,T]\times\R^d.$$
Then in view of \eqref{0521:01}, we have
\begin{align*}
\int_0^t\phi\left(\frac{s}{\eps},X_s^\eps\right)\dif s
&
=\int_0^t\phi_\eps(s,X_s^\eps)\dif s
=\eps\int_0^t\p_s\psi_\eps(s,X_s^\eps)\dif s\\
&
=\eps\left(\psi_\eps(t,X_t^\eps)-M_t^\eps-\int_0^t(b_\eps\cdot\nabla\psi_\eps)(s,X_s^\eps)\dif s
-\int_0^t\sL_{\sigma_\eps}\psi_\eps(s,X_s^\eps)\dif s\right),
\end{align*}
which implies that
\begin{align}\label{0713:001}
&
\mE\left(\sup_{t\in[0,T]}\left|\int_0^t\phi\left(\frac{s}{\eps},X_s^\eps\right)\dif s\right|^p\right)\nonumber\\
&
\leq\eps^p3^{p-1}\left[\mE\left(\sup_{t\in[0,T]}|\psi_\eps(t,X^\eps_t)|^p\right)
+\mE\left(\sup_{t\in[0,T]}|M_t^\eps|^p\right)+\sI\right]\\
&
\lesssim \eps^p\left[\|\psi_\eps\|_{L_T^\infty}^p
+\mE\left(\sup_{t\in[0,T]}|M_t^\eps|^p\right)+\sI\right],\nonumber
\end{align}
where
\begin{equation*}
\sI:=\mE\left(\sup_{t\in[0,T]}\left|\int_0^t(b_\eps\cdot\nabla\psi_\eps)(s,X_s^\eps)\dif s\right|^p\right)
+\mE\left(\sup_{t\in[0,T]}\left|\int_0^t\sL_{\sigma_\eps}\psi_\eps(s,X^\eps_s)\dif s\right|^p\right).
\end{equation*}
If $p\geq2$, then by the Burkholder-Davis-Gundy inequality one sees that
for any $\iota>0$
\begin{align}\label{0713:02}
\mE\left(\sup_{t\in[0,T]}|M_t^\eps|^p\right)
&
\lesssim \mE\left(\int_0^T\int_{\mR^d}\left|\psi_\eps(s,X^\eps_s+\sigma_\eps(s,X_s^\eps)z)
-\psi_\eps(s,X^\eps_s)\right|^2\nu(\dif z)\dif s\right)^{\frac{p}{2}}\nonumber\\
&\quad
+\mE\int_0^T\int_{\mR^d}\left|\psi_\eps(s,X^\eps_s+z)
-\psi_\eps(s,X^\eps_s)\right|^p
\nu(\dif z)\dif s\nonumber\\
&
\lesssim \left(\int_0^T\int_{\mR^d}
\|\psi_\eps(s)\|_\infty^2\wedge(|z|^{\alpha+\iota}
\|\psi_\eps(s)\|_{\bC^{(\alpha+\iota)/2}}^2)\nu(\dif z)\dif s\right)^{\frac{p}{2}}\\\nonumber
&\quad
+\int_0^T\int_{\mR^d}\left(
\|\psi_\eps(s)\|_\infty^p\wedge(|z|^{(\alpha+\iota)p/2}
\|\psi_\eps(s)\|_{\bC^{(\alpha+\iota)/2}}^p)\right)
\nu(\dif z)\dif s\\\nonumber
&
\lesssim
\|\psi_\eps\|_{L^\infty_T(\bC^{(\alpha+\iota)/2})}^p\left(
\left(\int_{\mR^d}(1\wedge|z|^{\alpha+\iota})\nu(\dif z)\right)^{\frac{p}{2}}
+\int_{\R^d}(1\wedge|z|^{\frac{(\alpha+\iota)p}{2}})
\nu(\dif z)\right).\nonumber
\end{align}
Now consider the case where $1\leq p<2$. Let
\begin{equation*}
Z_t:=\left|\int_0^t\int_{\mR^d}\left(\psi_\eps(s,X^\eps_{s-}+\sigma_\eps(s,X_{s-}^\eps)z)
-\psi_{\eps}(s,X^\eps_{s-})\right)\tilde{N}(\dif s,\dif z)\right|^2.
\end{equation*}
By Doob's maximal inequality and It\^o's isometry, $Z_t$ is  dominated by
\begin{equation*}
A_t:=\int_0^t\int_{\mR^d}\left|\psi_\eps(s,X^\eps_{s}+\sigma_\eps(s,X_{s}^\eps)z)
-\psi_{\eps}(s,X^\eps_{s})\right|^2\nu(\dif z)\dif s.
\end{equation*}
Hence it follows from Lemma \ref{Leneq} that for any $1\leq p<2$ and $\iota>0$,
\begin{align}\label{0713:03}
\mE\left(\sup_{t\in[0,T]}|M_t^\eps|^p\right)
&
=\mE\left(\sup_{t\in[0,T]}|Z_t|^{\frac{p}{2}}\right)\nonumber\\
&
\lesssim\mE\left(\int_0^T\int_{\mR^d}\left|\psi_\eps(s,X^\eps_s+\sigma_\eps(s,X_s^\eps)z)
-\psi_\eps(s,X^\eps_s)\right|^2\nu(\dif z)\dif s\right)^{\frac{p}{2}}\\
&
\lesssim\|\psi_\eps\|_{L^\infty_T(\bC^{(\alpha+\iota)/2})}^p
\left(\int_{\mR^d}(1\wedge|z|^{\alpha+\iota})\nu(\dif z)\right)^{\frac{p}{2}}.\nonumber
\end{align}
Finally, in view of the Krylov-type estimate \eqref{B03},  Lemma \ref{0116Lem02}
and Lemma \ref{lem0728}, we obtain that
for any $\delta\in(0,\frac{\alpha}{2}\wedge\beta)$ and
$\iota\in(0,2-\alpha)$,
\begin{align*}
\sI
&
=\mE\left(\sup_{t\in[0,T]}\left|\int_0^tb_\eps\cdot\nabla\psi_\eps(s,X_s^\eps)\dif s\right|^p\right)
+\mE\left(\sup_{t\in[0,T]}\left|\int_0^t\sL_{\sigma_\eps}\psi_\eps(s,X^\eps_s)\dif s\right|^p\right)\\
&
\lesssim\|b_\eps\cdot\nabla\psi_\eps\|_{L_T^\infty(\bC^{-\delta})}^p
+\|\sL_{\sigma_\eps}\psi_\eps\|_{L^\infty_T(\bC^{-(\alpha-\iota)/2})}^p\\
&
\lesssim \|b\|_{L^\infty(\R_+;\bC^{\beta})}^p\|\nabla\psi_\eps\|_{L_T^\infty(\bC^{-\delta})}^p
+\|\psi_\eps\|_{L_T^\infty(\bC^{(\alpha+\iota)/2})}\\
&
\lesssim
\|\psi_\eps\|_{L^\infty_T(\bC^{(-\delta+1)\vee[(\alpha+\iota)/2]})}^p
\end{align*}
which along with \eqref{0713:001}, \eqref{0713:02} and \eqref{0713:03}
completes the proof.
\end{proof}

Now we can prove Theorem \ref{Th1}.
\begin{proof}[Proof of Theorem \ref{Th1}]
Consider the following elliptic equation:
\begin{align}\label{Zvoneq}
\sL_{\bar{\sigma}} u-\lambda u+\bar{b}\cdot \nabla u+\bar{b}=0.
\end{align}
By Lemma \ref{ElLem24}, for $\lambda>0$ large enough, there exists
a unique solution $u$ to \eqref{Zvoneq} so that
for  $\delta\in(0,\beta-(1-\tfrac{\alpha}{2}))$,
\begin{equation}\label{0715:02}
\|\nabla u\|_\infty\leq \tfrac{1}{2}\ \ \text{and} \ \ \|u\|_{L^\infty_T(\bC^{1+\frac{\alpha}{2}+\delta})}
\lesssim\|u\|_{L^\infty_T(\bC^{\alpha+\beta})}<\infty.
\end{equation}
Define $\Phi(x):=x+u(x)$. Then $x\mapsto\Phi(x)$ is a $C^1$-diffeomorphism with
\begin{equation}\label{0403:04}
\|\nabla\Phi\|_\infty+\|\nabla\Phi^{-1}\|_\infty\leq4.
\end{equation}
It follows from \eqref{Main01}, It\^o's formula and \eqref{Zvoneq} that
\begin{align*}
u(X_t^\eps)=
&
u(x)
+\int_0^t\left(\sL_{\sigma_\eps(s,\cdot)}u(X_s^\eps)+b_\eps(s,X_s^\eps)\cdot\nabla u(X_s^\eps)\right)\dif s\\
&
+\int_0^t\int_{\R^d}\left(u(X_{s-}^\eps+\sigma_\eps(s,X_{s-}^\eps)z)-u(X_{s-}^\eps)\right)
\widetilde{N}(\dif s,\dif z)\\
=&
u(x)+\int_0^t\left(\sL_{\sigma_\eps(s,\cdot)}u(X_s^\eps)-\sL_{\bar{\sigma}}u(X_s^\eps)\right)\dif s
\\
&+\int_0^t\left(b_\eps(s,X_s^\eps)-\bar{b}(X_s^\eps)\right)
\cdot\nabla u(X_s^\eps)\dif s+\int_0^t(\lambda u-\bar b)(X_s^\eps)\dif s
\\
&
+\int_0^t\int_{\R^d}\left(u(X_{s-}^\eps+\sigma_\eps(s,X_{s-}^\eps)z)-u(X_{s-}^\eps)\right)
\widetilde{N}(\dif s,\dif z)
\end{align*}
and
\begin{align*}
u(\bar{X}_t)=
&
u(x)+\int_0^t(\lambda u-\bar b)(\bar{X}_s)\dif s
+\int_0^t\int_{\R^d}\left(u(\bar{X}_{s-}+\bar{\sigma}(\bar{X}_{s-})z)-u(\bar{X}_{s-})\right)
\widetilde{N}(\dif s,\dif z).
\end{align*}
Moreover, we have
$$
X^\eps_t-\bar X_t=\int_0^t\left(b_\eps(s,X_s^\eps)-\bar{b}(\bar X_s)\right)\dif s
+\int_0^t\int_{\R^d}(\sigma_\eps(s,X_{s-}^\eps)-\bar\sigma(\bar X_{s-}))z
\widetilde{N}(\dif s,\dif z).
$$
Let
\begin{align*}
Y_t^\eps:=\Phi(X_t^\eps),\ \  Y_t:=\Phi(\bar{X}_t).
\end{align*}
Combining the above calculations, we get
\begin{align}\label{0713:01}
Y_t^\eps-Y_t=\lambda \int_0^t\left(u(X_s^\eps)-u(\bar{X}_s)\right)\dif s+\sJ_1(t)+\sJ_2(t),
\end{align}
where
\begin{align*}
\sJ_1(t):=\int_0^t\left(b_{\eps}(s,X_s^\eps)-\bar{b}(X_s^\eps)\right)\cdot\nabla\Phi(X_s^\eps)\dif s
\end{align*}
and
\begin{align*}
\sJ_2(t)
&
:=\int_0^t\left(\sL_{\sigma_{\eps}(s,\cdot)}u(X_s^\eps)-\sL_{\bar{\sigma}}u(X_s^\eps)\right)\dif s
+\int_0^t\int_{\R^d}(\sigma_\eps(s,X_{s-}^\eps)-\bar\sigma(\bar X_{s-}))z
\widetilde{N}(\dif s,\dif z)\\
&\quad
+\int_0^t\int_{\R^d}\left(u(X_{s-}^\eps+\sigma_\eps(s,X_{s-}^\eps)z)-u(X_{s-}^\eps)
-u(\bar{X}_{s-}+\bar{\sigma}(\bar{X}_{s-})z)+u(\bar{X}_{s-})\right)
\widetilde{N}(\dif s,\dif z).
\end{align*}
Let $p\geq1$. By \eqref{0403:04}, \eqref{0715:02} and \eqref{0713:01}, we have
\begin{align}\label{Zvoneq05}
\mE\left(\sup_{t\in[0,T]}|X_t^\eps-\bar{X}_t|^p\right)
&
\leq 4^p
\mE\left(\sup_{t\in[0,T]}|Y_t^\eps-Y_t|^p\right)\\
&
\lesssim\int_0^T\mE\left(\sup_{s\in[0,t]}\left|X_s^\eps-\bar{X}_s\right|^p\right)\dif t
+\sum_{i=1}^2\mE\left(\sup_{t\in[0,T]}|\sJ_i(t)|^p\right).\nonumber
\end{align}

We first treat the term $\sI_1$.
For simplicity of notations, we write
$$
B(s,x):=b(s,x)\cdot\nabla\Phi(x),\ \bar{B}(x):=\bar{b}(x)\cdot\nabla\Phi(x),\ B_n:=B*\rho_n.
$$
For any $p\geq1$, we have
\begin{align}\label{0403:05}
\mE\left(\sup_{t\in[0,T]}|\sJ_1(t)|^p\right)
&
=\mE\left[\sup_{t\in[0,T]}\left|\int_0^t\left(b_{\eps}(s,X^\eps_s)-\bar{b}(X^\eps_s)\right)
\cdot\nabla\Phi(X_s^\eps)\dif s\right|^p\right]\nonumber\\
&
\lesssim \mE\left[\sup_{t\in[0,T]}\left|\int_0^t\left(B\left(\frac{s}{\eps},X^\eps_s\right)
-B_n\left(\frac{s}{\eps},X^\eps_s\right)\right)\dif s\right|^p\right]\nonumber\\
&\quad
+\mE\left[\sup_{t\in[0,T]}\left|\int_0^t\left(B_n\left(\frac{s}{\eps},X^\eps_s\right)
-\bar{B}_n(X^\eps_s)\right)\dif s\right|^p\right]\nonumber\\
&\quad
+\mE\left[\sup_{t\in[0,T]}\left|\int_0^t\left(\bar{B}_n(X^\eps_s)-\bar{B}(X^\eps_s)\right)
\dif s\right|^p\right]\no\\
&=:\sJ_{11}+\sJ_{12}+\sJ_{13}.
\end{align}
Then in view of Lemma \ref{0116Lem02},
\eqref{B03} and \eqref{UConAeq}, we have for any $\delta\in[0,\alpha/2)$,
\begin{align}\label{SI01}
\sJ_{11}+\sJ_{13}
&
\lesssim \|B-B_n\|^p_{L^\infty(\R_+;\bC^{-\delta})}
+\|\bar{B}-\bar{B}_n\|^p_{L^\infty(\R_+;\bC^{-\delta})}\\\nonumber
&
\lesssim n^{-p[(\delta+\beta\wedge(\alpha+\beta-1))\wedge1]}
\|b\|_{L^\infty(\mR_+;\bC^\beta)}\|u\|_{L^\infty_T\bC^{\alpha+\beta}} \lesssim n^{-p(\alpha\wedge1)}.
\end{align}
Thanks to Lemma \ref{0116Lem},
we obtain that for any $\kappa>(1-\alpha/2)\vee\alpha/2$,
\begin{align}\label{0306:04}
\sJ_{12}
&
\lesssim \eps^p\sup_{t\in[0,T]}\left\|\int_0^{\frac{t}{\eps}}
\left(B_n(s,\cdot)-\bar{B}_n(\cdot)\right)\dif s\right\|_{\bC^\kappa}^p\nonumber\\
&
\lesssim \eps^p\sup_{t\in[0,T]}\left\|\int_0^{\frac{t}{\eps}}
\left(B_n(s,\cdot)-\bar{B}_n(\cdot)\right)\dif s\right\|_{\bC^{\kappa\vee[\gamma\wedge(\alpha+\beta-1)]}}^p\\\nonumber
&
\lesssim n^{\tilde{\kappa}}\eps^p\sup_{t\in[0,T]}\left\|\int_0^{\frac{t}{\eps}}
\left(b(s,\cdot)-\bar b(\cdot)\right)\cdot\nabla\Phi(\cdot)\dif s\right\|_{\bC^{\gamma\wedge(\alpha+\beta-1)}}^p\\\nonumber
&
\lesssim n^{\tilde{\kappa}}\ell_1^p\left(\tfrac{T}{\eps}\right),
\end{align}
where $\tilde{\kappa}:=p\left(\kappa-[\gamma\wedge(\alpha+\beta-1)]\right)\vee0
=p[(\kappa-\gamma)\vee(\kappa-\alpha-\beta+1)]\vee0$.
Then \eqref{SI01} and \eqref{0306:04} yield that for any $p\geq1$,
\begin{equation}\label{0403:06}
\mE\left(\sup_{t\in[0,T]}\left|\sJ_1(t)\right|^p\right)
\lesssim n^{-p(\alpha\wedge1)}+n^{\tilde{\kappa}}\ell_1^p\left(\tfrac{T}{\eps}\right).
\end{equation}

Next we consider the term $\sI_2$ for $\alpha\in(0,1]$ and $\alpha\in(1,2)$ seperately.

{\bf (Case 1: $\alpha\in(0,1]$)}
In this case, $\sigma_\eps(s,x)=\bar{\sigma}(x)=c\mI$.
Without loss of generality we assume that $c=1$. By definition we can write
\begin{align*}
\sJ_2(t)
&
=\int_0^t\int_{\R^d}\left(\delta_zu(X_{s-}^\eps)-\delta_zu(\bar{X}_{s-})\right)
\widetilde{N}(\dif s,\dif z),
\end{align*}
where
$$
\delta_zu(x):=u(x+z)-u(x).
$$
For any $x_1,x_2\in\R^d$, by \eqref{0715:02} we have
\begin{align}\label{0715:03}
\left|\delta_zu(x_1)-\delta_zu(x_2)\right|
&
=\left|(x_1-x_2)\cdot\int_0^1\delta_z\nabla u(\theta x_1+(1-\theta) x_2)\dif\theta\right|\no\\
&
\leq |x_1-x_2|\left(\|\nabla u\|_\infty\wedge (|z|^{\frac{\alpha}{2}+\delta}
\|\nabla u\|_{\bC^{\frac{\alpha}{2}+\delta}})\right)\\
&
\lesssim |x_1-x_2|\left(1\wedge|z|^{\frac{\alpha}{2}+\delta}\right).\no
\end{align}
Therefore, by Lemma \ref{Leneq} and \eqref{0715:03}, one sees that for any $1\leq p<2$,
\begin{align}\label{0714:02}
\mE\left(\sup_{t\in[0,T]}|\sJ_2(t)|^p\right)
&
\lesssim \mE\left(\int_0^T\int_{\mR^d}|\delta_zu(X_s^\eps)
-\delta_zu(\bar{X}_s)|^2\nu(\dif z)\dif s\right)^{\frac{p}{2}}\no\\
&
\lesssim \mE\left(\int_0^T|X_s^\eps-\bar{X}_s|^2\int_{\mR^d}1\wedge|z|^{\alpha+2\delta}\nu(\dif z)\dif s\right)^{\frac{p}{2}}\no\\
&
\lesssim \mE\left(\int_0^T|X_s^\eps-\bar{X}_s|^2\dif s\right)^{\frac{p}{2}}
\leq  \mE\left(\sup_{s\in[0,T]}|X_s^\eps-\bar{X}_s|\int_0^T|X_s^\eps-\bar{X}_s|\dif s\right)^{\frac{p}{2}}\no\\
&
\le\frac12\mE\left(\sup_{t\in[0,T]}|X_t^\eps-\bar{X}_t|^p\right)
+C\int_0^T\mE\left(\sup_{t\in[0,s]}|X_t^\eps-\bar{X}_t|^p\right)\dif s.
\end{align}
If $p\geq2$, then it follows from the Burkholder-Davis-Gundy inequality
and \eqref{0715:03} that
\begin{align}\label{0715:04}
\mE\left(\sup_{t\in[0,T]}|\sJ_2(t)|^p\right)
&
\lesssim \mE\left(\int_0^T\int_{\mR^d}|\delta_zu(X_s^\eps)
-\delta_zu(\bar{X}_s)|^2\nu(\dif z)\dif s\right)^{\frac{p}{2}}\no\\
&\quad
+\mE\int_0^T\int_{\R^d}\left|\delta_zu(X_s^\eps)-\delta_zu(\bar{X}_s)\right|^p\nu(\dif z)\dif s\no\\
&
\lesssim \mE\left(\int_0^T|X_s^\eps-\bar{X}_s|^2\int_{\mR^d}
1\wedge |z|^{\alpha+2\delta}\nu(\dif z)\dif s\right)^{\frac{p}{2}}\\
&\quad
+\mE\int_0^T |X_s^\eps-\bar{X}_s|^p\int_{\R^d}1\wedge |z|^{\frac{p\alpha}{2}+p\delta}
\nu(\dif z)\dif s\no\\
&
\le C\int_0^T\mE\left(\sup_{t\in[0,s]}|X_t^\eps-\bar{X}_t|^p\right)\dif s.\no
\end{align}

{\bf (Case 2: $\alpha\in(1,2)$)} In this case, we can write
\begin{align*}
\sJ_2(t)
&
=\int_0^t\left(\sL_{\sigma_{\eps}(s,\cdot)}u(X_s^\eps)-\sL_{\bar{\sigma}}u(X_s^\eps)\right)\dif s
+\int_0^t\int_{|z|>1}\left(\sigma_\eps(s,X_{s-}^\eps)z-\bar{\sigma}(\bar{X}_{s-})z\right)
N(\dif s,\dif z)\\
&\quad
+\int_0^t\int_{|z|\leq1}\left(\sigma_\eps(s,X_{s-}^\eps)z-\bar{\sigma}(\bar{X}_{s-})z\right)
\widetilde{N}(\dif s,\dif z)\\
&\quad
+\int_0^t\int_{\R^d}\left(u(X_{s-}^\eps+\sigma_{\eps}(s,X_{s-}^\eps)z)-u(X_{s-}^\eps)
-u(\bar{X}_{s-}+\bar{\sigma}(\bar{X}_{s-})z)+u(\bar{X}_{s-})\right)\widetilde{N}(\dif s,\dif z)\\
&
=:\sJ_{21}(t)+\sJ_{22}(t)+\sJ_{23}(t)+\sJ_{24}(t),
\end{align*}
which implies that
\begin{equation}\label{0714:01}
\mE\left(\sup_{t\in[0,T]}|\sJ_{2}(t)|^p\right)
\leq 4^{p-1}\sum_{i=1}^4\mE\left(\sup_{t\in[0,T]}|\sJ_{2i}(t)|^p\right).
\end{equation}
By the definitions of $\sL_{\sigma_\eps}$ and $\sL_{\bar{\sigma}}$, we obtain
\begin{align}\label{0114:01}
\sJ_{21}(t)
&
=\int_0^t\left(\sL_{\sigma_{\eps}(s,\cdot)}u(X_s^\eps)-\sL_{\bar{\sigma}}u(X_s^\eps)\right)\dif s\nonumber\\
&
=\int_0^t\int_{|z|>1}\Gamma_1(z,s)\nu(\dif z)\dif s
+\int_0^t\int_{|z|\leq1}\Gamma_2(z,s)\nu(\dif z)\dif s,
\end{align}
where
\begin{align*}
\Gamma_1(z,s)&:=u(X_s^\eps+\sigma_{\eps}(s,X_s^\eps)z)
-u(X_s^\eps+\bar{\sigma}(X_s^\eps)z)\\
\Gamma_2(z,s)&:=\Gamma_1(z,s)-z(\sigma_{\eps}(s,X_s^\eps)
-\bar{\sigma}(X_s^\eps))\cdot\nabla u(X_s^\eps).
\end{align*}
Since $\alpha\in(1,2)$, by Taylor's expansion, we have
\begin{align}\label{0728:03}
\left|\int_0^t\int_{|z|>1}\Gamma_1(z,s)\nu(\dif z)\dif s\right|
&
\lesssim \|\nabla u\|_\infty\int_0^t|\sigma_{\eps}(s,X_s^\eps)-\bar{\sigma}(\bar{X}_s^\eps)|\dif s\int_{|z|>1}|z|\nu(\dif z),
\end{align}
and
\begin{align}\label{0114:02}
\left|\int_0^t\int_{|z|\leq1}\Gamma_2(z,s)\nu(\dif z)\dif s\right|
&
\lesssim \|u\|_{L^\infty_T(\bC^{\alpha+\beta})}\int_0^t|\sigma_{\eps}(s,X_s^\eps)-\bar{\sigma}(X_s^\eps)|\dif s\int_{|z|\leq1}|z|^{\alpha+\beta}\nu(\dif z).
\end{align}
And we also have
which along with  \eqref{0114:01}, \eqref{0728:03}, \eqref{0114:02}, Lemma \ref{0116Lem}
and {\bf(A$_\sigma$)} implies that for any $p\geq1$,
\begin{align}\label{0114:04}
\mE\left(\sup_{t\in[0,T]}|\sJ_{21}(t)|^p\right)=
&
\mE\left(\sup_{t\in[0,T]}\left|\int_0^t\left(\sL_{\sigma_{\eps}(s,\cdot)}u(X_s^\eps)
-\sL_{\bar{\sigma}}u(X_s^\eps)\right)\dif s\right|^p\right)\nonumber\\\nonumber
&
\lesssim \mE\left(\int_0^T\left|\sigma_{\eps}(s,X_s^\eps)-\bar{\sigma}(X_s^\eps)\right|\dif s
\right)^p\\
&
\lesssim \mE\left(\int_0^T\left|\sigma_{\eps}(s,X_s^\eps)-\bar{\sigma}(X_s^\eps)\right|^2\dif s
\right)^{\frac{p}{2}}\\\nonumber
&
\lesssim \left(\eps\sup_{t\in[0,T]}\left\|\int_0^{\frac{t}{\eps}}
\left|\sigma(s,\cdot)-\bar{\sigma}(\cdot)\right|^{2}\dif s\right\|_{\bC^1}
\right)^{\frac{p}{2}}\\\nonumber
&
\lesssim \left(\ell_2\left(\tfrac{T}{\eps}\right)\right)^{\frac{p}{2}}.
\end{align}

By \cite[Lemma 2.3]{SZ15}, Lemma \ref{0116Lem} and {\bf(A$_\sigma$)},
one sees that for any $1\leq p<\alpha$,
\begin{align}\label{0402:01}
\mE\left(\sup_{t\in[0,T]}|\sJ_{22}(t)|^p\right)
&
=\mE\left(\sup_{t\in[0,T]}\left|\int_0^t\int_{|z|>1}
\left(\sigma_{\eps}(s,X_{s-}^\eps)-\bar{\sigma}(\bar{X}_{s-})\right)z
N(\dif s,\dif z)\right|^p\right)\nonumber\\\nonumber
&
\lesssim \mE\left(\int_0^T\int_{|z|>1}
\left|\left(\sigma_{\eps}(s,X_s^\eps)
-\bar{\sigma}(\bar{X}_s)\right)z
\right|\nu(\dif z)\dif s\right)^p\\
&\quad
+\mE\int_0^T\int_{|z|>1}
\left|\left(\sigma_{\eps}(s,X_s^\eps)
-\bar{\sigma}(\bar{X}_s)\right)z
\right|^p\nu(\dif z)\dif s\\\nonumber
&
\lesssim \left(\mE\int_0^T\left|\sigma_{\eps}(s,X_s^\eps)-\bar{\sigma}(X_s^\eps)\right|^2\dif s\right)^{\frac{p}{2}}
+\int_0^T\mE\left(\sup_{s\in[0,t]}|X_s^\eps-\bar{X}_s|^p\right)\dif t\\\nonumber
&
\lesssim \left(\ell_2\left(\tfrac{T}{\eps}\right)\right)^{\frac{p}{2}}
+\int_0^T\mE\left(\sup_{s\in[0,t]}|X_s^\eps-\bar{X}_s|^p\right)\dif t.
\end{align}
It follows from Lemma \ref{Leneq},
Young's inequality, Lemma \ref{0116Lem} and {\bf(A$_\sigma$)} that for any $p\in(0,2)$,
\begin{align}\label{0402:02}
\mE\left(\sup_{t\in[0,T]}|\sJ_{23}(t)|^p\right)
&
=\mE\left(\sup_{t\in[0,T]}\left|\int_0^t\int_{|z|\leq1}
\left(\sigma_{\eps}(s,X_{s-}^\eps)-\bar{\sigma}(\bar{X}_{s-})\right)z
\widetilde{N}(\dif s,\dif z)\right|^p\right)\nonumber\\\nonumber
&
\lesssim \mE\left(\int_0^T\int_{|z|\leq1}
\left|\left(\sigma_{\eps}(s,X_s^\eps)
-\bar{\sigma}(\bar{X}_s)\right)z
\right|^2\nu(\dif z)\dif s\right)^{\frac{p}{2}}\\
&
\lesssim \mE\left(\int_0^T\left|\sigma_{\eps}(s,X_s^\eps)
-\bar{\sigma}(X_s^\eps)\right|^2
+|X_s^\eps-\bar{X}_s|^2\dif s\right)^{\frac{p}{2}}\\\nonumber
&
\le \frac{1}{2}\mE\left(\sup_{t\in[0,T]}
|X_s^\eps-\bar{X}_s|^p\right)
+C\int_0^T\mE|X_s^\eps-\bar{X}_s|^p\dif s\\\nonumber
&\quad
+C\mE\left(\int_0^T\left|\sigma_{\eps}(s,X_s^\eps)
-\bar{\sigma}(X_s^\eps)\right|^2\dif s\right)^{\frac{p}{2}}\\\nonumber
&
\le \frac{1}{4}\mE\left(\sup_{t\in[0,T]}
|X_s^\eps-\bar{X}_s|^p\right)
+C\int_0^T\mE\left(\sup_{s\in[0,t]}|X_s^\eps-\bar{X}_s|^p\right)\dif t
+C\left(\ell_2\left(\tfrac{T}{\eps}\right)\right)^{\frac{p}{2}}.
\end{align}
The last term is $\sJ_{24}$. Define
$$u_{\sigma_{\eps}}^z(s,x):=u(x+\sigma_{\eps}(s,x)z), \quad
u_{\bar{\sigma}}^z(x):=u(x+\bar{\sigma}(x)z), \quad \forall (t,x,z)\in\R_+\times\R^{2d}.$$
We note that
\begin{align*}
\nabla_x(u_{\bar\sigma}^z-u)(x)
&
=\nabla u(x+\bar{\sigma}(x)z)(\mI+\nabla\bar{\sigma}(x)z)-\nabla u(x)\\
&
=\delta_{\bar{\sigma}(x)z}\nabla u(x)+\nabla u(x+\bar{\sigma}(x)z)\nabla\bar{\sigma}(x)z,
\end{align*}
which implies that 
\begin{align}\label{0114:05}
&
\quad |u_{\sigma_{\eps}}^z(s,X_s^\eps)-u(X_s^\eps)-u_{\bar{\sigma}}^z(\bar{X}_s)+u(\bar{X}_s)| \no
\\
&
\le\left|u_{\sigma_{\eps}}^z(s,X_s^\eps)-u_{\bar{\sigma}}^z(X_s^\eps)\right|
+\left|\left(u_{\bar{\sigma}}^z(X_s^\eps)-u(X_s^\eps)\right)
-\left(u_{\bar{\sigma}}^z(\bar{X}_s)-u(\bar{X}_s)\right)\right|
\\
&
\lesssim \|\nabla u\|_{\infty}|\sigma_{\eps}(s,X_s^\eps)-\bar{\sigma}(X_s^\eps)||z|
+|X^\eps_s-\bar{X}_s|\left[\|\nabla u\|_\infty\wedge\left(|z|^{\alpha/2+\delta}
\|u\|_{L^\infty_T\bC^{1+\alpha/2+\delta}}\right)\right]\no\\
&\ +|X^\eps_s-\bar{X}_s|\|\nabla u\|_\infty\|\nabla \sigma\|_\infty|z|.\no
\end{align}
In view of \eqref{0114:05}, we obtain that for any $p\in[1,\alpha)$,
\begin{align}\label{0403:07}
&
\mE\left(\sup_{t\in[0,T]}|\sJ_{24}(t)|^p\right)\nonumber\\\nonumber
&
\lesssim \mE\left(\int_0^T\int_{|z|\leq1}
\left(|\sigma_{\eps}(s,X_s^\eps)-\bar{\sigma}(X_s^\eps)|^2
+|X_s^\eps-\bar{X}_s|^2\right)|z|^{\alpha+2\delta}
\nu(\dif z)\dif s\right)^{\frac{p}{2}}\\
&\quad
+\mE\int_0^T\int_{|z|>1}\left(|\sigma_{\eps}(s,X_s^\eps)-\bar{\sigma}(X_s^\eps)|^p
+|X_s^\eps-\bar{X}_s|^{p}\right)|z|^p\nu(\dif z)\dif s\\\nonumber
&
\le\frac12\mE\left(\sup_{t\in[0,T]}|X_t^\eps-\bar{X}_t|^p\right)+C \mE\left(\int_0^T\left|\sigma_{\eps}(s,X_s^\eps)-\bar{\sigma}(X_s^\eps)\right|^2\dif s\right)^{\frac{p}{2}}
+C\mE\int_0^T\left|X_s^\eps-\bar{X}_s\right|^p\dif s\\\nonumber
&
\le\frac14\mE\left(\sup_{t\in[0,T]}|X_t^\eps-\bar{X}_t|^p\right)+C \int_0^T\mE\left(\sup_{s\in[0,t]}\left|X_s^\eps-\bar{X}_s\right|^p\right)\dif t
+C\left(\ell_2\left(\tfrac{T}{\eps}\right)\right)^{\frac{p}{2}}.
\end{align}
Therefore, by \eqref{0714:01}, \eqref{0114:04}, \eqref{0402:01},
\eqref{0402:02} and \eqref{0403:07}, one sees that
\begin{align}\label{0714:03}
\mE\left(\sup_{t\in[0,T]}|\sJ_2(t)|^p\right)
&
\leq \frac{1}{2}\mE\left(\sup_{t\in[0,T]}
|X_s^\eps-\bar{X}_s|^p\right)
+C\left(\ell_2\left(\tfrac{T}{\eps}\right)\right)^{\frac{p}{2}}\\
&\quad
+C\int_0^T\mE\left(\sup_{s\in[0,t]}|X_s^\eps-\bar{X}_s|^p\right)\dif t.\no
\end{align}

Combining \eqref{Zvoneq05}, \eqref{0403:06}, \eqref{0714:02}, \eqref{0715:04},
\eqref{0714:03} and Gronwall's inequality,
we complete the proof by letting $n=\ell_1^{-\frac{p}{p(1\wedge\alpha)+\tilde{\kappa}}}(\frac{T}{\eps})$
for any $1\leq p<\alpha$.
\end{proof}

\subsection{Proof of Theorem \ref{Th2}}\label{WCR}

Let $T>0$, $b\in L_T^\infty(\bC^\beta)$,
$\sigma$ satisfies {\bf(H$_{\sigma}$)},
and $\{L_t^{(\alpha)},t\geq0\}$ is an $\alpha$-stable process
with $(\alpha,\beta)\in\cA$.
Consider the following SDE:
\begin{equation}\label{1203:01}
\dif X_t=b(t,X_t)\dif t+\sigma(t,X_t)\dif L_t^{(\alpha)}.
\end{equation}
and corresponding backward nonlocal PDE for $0\leq t\leq T$:
\begin{equation}\label{1203:02}
  \partial_tu+\sL_\sigma u+b\cdot\nabla u+f=0,\quad u(T,x)=0
\end{equation}
for $f\in L_T^\infty(\bC^\iota)$ with $\iota>0$.
We use $\bD:=D([0,T];\R^d)$ to denote the space of all c\`adl\`ag functions
$\phi:[0,T]\rightarrow\R^d$ with Skorokhod $J_1$-topology.
Then $\bD$ is a Polish space with the usual Borel $\sigma$-algebra $\sB(\bD)$.
The mappings $\pi_t:\bD\rightarrow\R^d,~t\in[0,T]$, defined by
$$
\pi_t(\omega)=\omega_t,\quad \forall \omega\in\bD,
$$
are called coordinate process.
Define $\cD_t:=\sB(D([0,t];\R^d))$ for all $t\in[0,T]$, which is equal to the natural filtration.
Denote by $\cP(\bD)$ the space of probability measures over $\bD$.
Let $\|\cdot\|_{\var}$ be the total variation distance on $\cP(\bD)$ defined by
\begin{equation*}
\|\mP_1-\mP_2\|_{\var}:=\sup_{A\in\sB(\bD)}|\mP_1(A)-\mP_2(A)|
\end{equation*}
for $\mP_1,\mP_2\in\cP(\bD)$. It can be verified that
$\cW_1(\mP_1,\mP_2)\leq \|\mP_1-\mP_2\|_{\var}$.

Now we recall the following definition of martingale solution
in the sense of Either and Kurtz \cite{EK86}.

\begin{definition}\rm\label{MSolution}
Let $\alpha\in(0,2)$ and $T>0$. For $x\in\R^d$, a probability measure $\mP_x\in\cP(\bD)$
is called {\it martingale solution} of SDE \eqref{1203:01} starting at $x$ if
\begin{itemize}
  \item[(i)] $\mP(\pi_0=x)=1$;
  \item[(ii)] For any $f\in L_T^\infty(\bC^\iota)$ with $\iota>0$, the process
  \begin{equation*}
  M_t:=u(t,\pi_t)-u(0,x)+\int_0^tf(r,\pi_r)\dif r
  \end{equation*}
  is a martingale under $\mP_x$ with respect to filtration $(\cD_t)_{t\in[0,T]}$,
  where $u$ is the unique solution of PDE \eqref{1203:02}.
\end{itemize}
The set of all the martingale solutions to \eqref{1203:01}
starting from $x$ is defined by $\cM(x)$.
\end{definition}

Note that Definition \ref{MSolution} is a generalization of the classical notion
of a weak solution. Indeed, it can be shown that $\{X_t^n,t\geq0\}$ is a weak solution to
\begin{equation}\label{0407:01}
\dif X_t^n=b_n(t,X_t^n)\dif t+\sigma(t,X_t^n)\dif L_t^{(\alpha)},
\quad X_0^n=x\in\R^d
\end{equation}
if and only if $\mP^n_x:=\mP\circ\left(X^n(x)\right)^{-1}$ is the martingale
solution of \eqref{0407:01} in the sense of Definition \ref{MSolution}.
Assume that $(\alpha,\beta)\in\cA_2$. Then it follows from \cite{CZZ22}, \cite{LZ22}, \cite[Theorem 1.1]{WH23}
and \cite[Theorem 2.2]{SX23} that
there exists a unique martingale solution $\mP_x\in\cM(x)$ to \eqref{1203:01}.
Moreover, by denoting $\mP(t):=\mP\circ(\pi_t)^{-1}\in \cP(\mR^d)$, based on \cite[Theorem 4.6]{WH23} and \cite[Theorem 4.1]{HRW24}, we have when {\bf $\alpha\in(1,2)$} for any $\beta'\in(-\frac{\alpha-1}{2},\beta)$
\begin{align}\label{0513:01}
\sup_{t\in[0,T]}\sup_{x\in\mR^d}\|\mP_x(t)-\mP^n_x(t)\|_{\var}
\lesssim \|b_n-b\|_{L^\infty(\mR_+;\bC^{\beta'})}
\lesssim n^{-(\beta-\beta')};
\end{align}
when {\bf $\alpha\in(0,1]$},
\begin{align}\label{0513:02}
\sup_{t\in[0,T]}\sup_{x\in\mR^d}\|\mP_x(t)-\mP^n_x(t)\|_{\cW_1}
\lesssim \sup_{t\in\R_+}\|b_n-b\|_{\infty}
\lesssim n^{-\beta}.
\end{align}

Now it is a position to prove Theorem \ref{Th2}.
\begin{proof}[Proof of Theorem \ref{Th2}]
Let $X_t^{n,\eps}$ be the unique strong solution to
\begin{equation*}
\dif X_t^{n,\eps}=b_n\left(\frac{t}{\eps},X_t^{n,\eps}\right)\dif t
+\sigma\left(\frac{t}{\eps},X_t^{n,\eps}\right)\dif L_t^{(\alpha)},
\end{equation*}
and $\bar{X}_t^n$ be the unique strong solution to
\begin{equation*}
\dif \bar{X}_t^{n}=\bar{b}_n(\bar{X}_t^{n})\dif t
+\bar{\sigma}(\bar{X}_t^n)\dif L_t^{(\alpha)}.
\end{equation*}
Then it follows from \eqref{0513:01} and \eqref{0513:02} that
\begin{align}\label{0513:03}
\sup_{\eps\ge0}\sup_{t\in[0,T]}\|\sL(X_t^{n,\eps})-\sL(X_t^{\eps})\|_{\var}\lesssim n^{-\delta}
\end{align}
with some $\delta>0$, where $X_t^{n,0}:=\bar{X}_t^{n}$ and $X_t^{0}:=\bar{X}_t$.

Let $\lambda>0$ and $n\geq1$.
Consider the following elliptic equations:
\begin{equation}\label{0308:01}
\sL_\sigma u_\lambda^n-\lambda u_\lambda^n+\bar{b}_n\cdot\nabla u_\lambda^n+\bar{b}_n=0.
\end{equation}
By Lemma \ref{ElLem24}, one sees that for any $n\geq1$ there exists a unique solution $u_\lambda^n$
to \eqref{0308:01} satisfying that for any $\eta\in[0,\alpha]$,
\begin{equation}
\|u_\lambda^n\|_{\bC^{\eta+\beta}}\lesssim (1+\lambda)^{-\frac{\alpha-\eta}{\alpha}}\|\bar{b}_n\|_{\bC^\beta}
\lesssim (1+\lambda)^{-\frac{\alpha-\eta}{\alpha}}\|\bar{b}\|_{\bC^\beta}.
\end{equation}
Take $\lambda>0$ large enough so that $\Phi_n(x):=x+u_n(x)$ is a $C^1$-diffeomorphism with
\begin{equation*}
\|\nabla\Phi_n\|_\infty+\|\nabla\Phi_n^{-1}\|_\infty\leq4.
\end{equation*}
Let $p\geq1$ when $\alpha\in(0,1]$ and $1\leq p<\alpha$ when $\alpha\in(1,2)$.
Hence it follows from Zvonkin's transform and Lemma \ref{0116Lem} that
\begin{align}\label{1218:05}
\mE\left(\sup_{t\in[0,T]}|X_t^{n,\eps}-\bar{X}_t^n|^p\right)
&
\lesssim_n \int_0^T\mE\left(\sup_{s\in[0,t]}|X_s^{\eps,n}-\bar{X}^n_s|^p\right)\dif t
+\left(\ell_2\left(\tfrac{T}{\eps}\right)\right)^{\frac{p}{2}}
\\\nonumber
&
+\mE\left[\sup_{t\in[0,T]}\left|\int_0^t\left(b_n\left(\frac{s}{\eps},X_s^{n,\eps}\right)
-\bar{b}_n(X_s^{n,\eps})\right)\cdot\nabla\Phi_n(X_s^{n,\eps})\dif s\right|^p\right]\\\nonumber
&
\lesssim_n \int_0^T\mE\left(\sup_{s\in[0,t]}|X_s^{\eps,n}-\bar{X}^n_s|^p\right)\dif t
+n^{(\kappa-\gamma)p}\ell_1\left(\tfrac{T}{\eps}\right)^p
+\left(\ell_2\left(\tfrac{T}{\eps}\right)\right)^{\frac{p}{2}},
\end{align}
where $\kappa$ is as in Lemma \ref{0116Lem}, which implies that for each $n\in\mN$,
\begin{align*}
\lim_{\eps\to0}\sup_{t\in[0,T]}\sup_{x\in\R^d}
\cW_1(\sL(X_s^{\eps,n}),\sL(\bar{X}_s^{n}))=0.
\end{align*}
Thus, in view of \eqref{0513:03}, we have
\begin{align*}
&
\lim_{\eps\to0}\sup_{t\in[0,T]}
\sup_{x\in\R^d}\cW_1\left(\mP_x^\eps(t),\mP_x(t)\right)\\
&
\le \lim_{\eps\to0}\sup_{t\in[0,T]}\sup_{x\in\R^d}
\cW_1(\sL(X_s^{\eps,n}),\sL(\bar{X}_s^{n}))
+\lim_{n\to\infty}\sup_{\eps\ge0}\sup_{t\in[0,T]}
\cW_1(\mP_x^{\eps,n}(s),\mP_x^{\eps}(s))\\
&
\quad
+\lim_{n\to\infty}\sup_{\eps\ge0}\sup_{t\in[0,T]}
\sup_{x\in\R^d}\cW_1(\mP_x^n(s),\mP_x(s))\\
&
\le \lim_{n\to\infty}\sup_{\eps\ge0}\sup_{t\in[0,T]}
\sup_{x\in\R^d}\cW_1(\sL(X_s^{\eps,n}),\sL(X_s^{\eps}))=0
\end{align*}
and complete the proof.
\end{proof}

\end{document}